\documentclass[11pt, a4paper, UKenglish]{article}
\usepackage[UKenglish]{babel}

\usepackage{preamble}

\author{Aleksandra Korzhenkova, Avelio Sep\'ulveda}
\date{\vspace{-5ex}}
\title{Vector-valued Gaussian free field conditioned to avoid a ball: Entropic repulsion of the norm and Freezing of spins}

\begin{document}
\maketitle

\pagenumbering{arabic}
\begin{abstract} 
 We study the laws of the two-dimensional vector-valued Dirichlet Gaussian free field and its massive lattice counterpart, conditioned to avoid a ball at every site of a subdomain. We  prove that, under this conditioning, the norm of the massless field exhibits entropic repulsion, while its angular components freeze at all mesoscopic scales. A key step in the analysis is showing that around any given point in the bulk of the range, the unconditioned field has no holes.

In the massive case, the conditioned field behaves differently: its norm remains uniformly bounded as the system size grows, leading to the existence of infinite-volume Gibbs measures. Furthermore, in the scalar massive case, the system undergoes a phase transition in the size of the avoided interval: for small intervals, the system admits a unique infinite-volume limit, while for sufficiently large centered intervals, multiple such limits exist. 
\end{abstract}



\section{Introduction}
\label{sec:intro} 
The behavior of spin systems is one of the subjects of study in probability theory and statistical physics. One of the most fundamental models in this context is the spin $O(N)$ model in two dimensions. This is a Gibbs measure on random functions from a graph $G \subseteq \Z^2$ to $\Ss^{N-1}$, whose energy at an edge is given by minus the inner product of the adjacent spins. Notable cases include the Ising model ($N=1$), the XY model ($N=2$), and the Heisenberg model ($N=3$).

The existence of a phase transition in the spin $O(N)$ model depends crucially on the parameter $N$. For $N=1$, the model undergoes a classical order–disorder phase transition: at high temperature it decorrelates exponentially, while at low temperature it is largely ordered, with spins at arbitrarily large distances likely to be aligned. For $N=2$, there is no classical phase transition, but instead a topological one: the system transitions between weak and strong disorder \cite{FS, vEL, AHPS}. In this case, correlations decay exponentially at high temperature, whereas at low temperature they decrease only at a polynomial rate. For $N \geq 3$, the low-temperature behavior remains a mystery. It is known that the model does not exhibit a classical phase transition and decorrelates exponentially at high temperature, but the rate of correlation decay in the low-temperature regime remains unknown. A key conjecture of Polyakov \cite{POLYAKOV197579} suggests that exponential decorrelation holds at all temperatures, which would imply the absence of a phase transition altogether.

The objective of this paper is to initiate a program to study the low-temperature behavior of the spin $O(N)$ model using the (massive) Gaussian free field (GFF) conditioned to have a norm, at every point, greater than a given value. The GFF is a generalization of the Gaussian random walk, with time replaced by a graph $G$. As a Gaussian vector, the GFF is analytically tractable: for instance, its $n$-point correlation functions can be computed explicitly in terms of the Green’s function of the underlying graph.
The key result justifying our approach is that the law of the GFF conditioned on its norm is that of a spin $O(N)$ model in an inhomogeneous environment. More precisely, let $\phi: G \rightarrow \R^N$ be an $N$-vectorial (massive) GFF—that is, a collection of $N$ independent scalar (massive) GFFs. Then conditionally on the norms $(|\phi(x)|)_{x \in G}$, the law of the angular component $\phi(x)/|\phi(x)|$ is that of the spin $O(N)$ model on $G$ with the inverse temperature at each edge $\{x,y\}$ given by $|\phi(x)|| \phi(y)|$. This connection has already proved fruitful: for $N=1$, it was used to advance the understanding of the Ising model on general graphs \cite{LW,DCRS}; for $N=2$, it was employed in \cite{AGS} to study the disordered environments, in which the $XY$ model mixes exponentially fast, even when the temperature is low in most locations. 

To exploit this connection for a spin $O(N)$ model at low temperature, we propose conditioning a (vector-valued) GFF to have norm at least $R > 0$ at every vertex. We show (Theorem \ref{thm:repulsion}) that this conditioning induces the so-called entropic repulsion phenomenon (cf. \cite{entropRep3+D, entropRep2D,FHO1,FHO2}) for the norm of the GFF: although it is not energetically efficient, the norm is pushed away from the origin so that typical fluctuations of the field do not violate the constraint. At the same time, we prove (Theorem \ref{thm:freezing}) that the spins become totally ordered at mesoscopic scales—that is, with high probability, for two typical points taken at arbitrary mesoscopic distances, the corresponding spins of the conditioned field point almost in the same direction.

The behavior explained above is a signal that to better mirror the properties of the spin $O(N)$ model at low temperature, it is necessary to add a mass to the conditioned field. That is to say, one needs to study the massive vector-valued GFF conditioned to have its norm bigger than $R$ everywhere. In this case, we show (Theorem \ref{thm:massiveGFF}) that the typical norm of the conditioned massive field at any given site remains bounded as $G \nearrow \Z^2$, and thus there are infinite-volume limits of this conditioned field.

We expect that future investigation of these infinite-volume measures will shed light on the geometric properties of the low-temperature regime of the spin $O(N)$ model. In particular, we believe that the isomorphism theorems connecting this GFF to loop soups may provide insight into the correlations of the spins of the spin $O(N)$ model in this regime.

\subsection*{Overview of main results}
\subsection{Massless case} 
Let $\phi \sim \P_n$ be an $N$-vectorial Dirichlet Gaussian free field\footnote{We normalize the field so that the variance of a coordinate at the origin is $\E[(\phi^1(0))^2]\sim \log n$. See Subsection \ref{subsec:notation&defs}, particularly Remark \ref{rem:normalization}.} on $\Lambda_n = [-n/2, n/2]^2 \cap \Z^2$. For an open ball $I = \B^N(0,R) \subset \R^N$ and $V \subset \Lambda_n$, define the event 
\begin{align*}
    \Omega_{V}^I \coloneqq \{\phi(x) \notin I, \; \forall x \in V\}.
\end{align*}
Our goal is to understand the behavior of the field upon conditioning on $\Omega_V^I$ as $V, \Lambda_n \nearrow \Z^2$. To avoid the effect of the Dirichlet boundary condition, we choose $V = D_n= n D \cap \Z^2$, where $D\subset \Lambda:= [-1/2, 1/2]^2$ is an open set containing $0$ with smooth outer boundary at positive distance from $\partial \Lambda$.

As mentioned earlier, the norm of the field ($|\phi|$) conditioned on $\Omega_{D_n}^I$, for any $R > 0$, behaves similarly to a scalar field conditioned to remain positive (the so-called \emph{hard-wall} condition)—see Section \ref{sec:repulsion}. A key ingredient in the analysis of the latter was the study of the (unconditioned) field's maximum. In the vectorial setting—perhaps unsurprisingly—the corresponding step involves analyzing the range of the (unconditioned) vector-valued field, rather than just its maximum.
\begin{theorem} \label{intro:thm:2Dno_hole}
    Let $\varepsilon > 0$ and $0<r_n<1$  such that $\log (1/r_n) = o(\log(n))$. Then,
    \begin{align*} 
        \sup \Big\{\P_n\Big[\phi(x) + t \notin I, \;\forall x \in \Lambda_n \cap \mathcal{B}^2(0, r_n n)\Big]: t\in \B^N \big(0,  2(1-\varepsilon)\log n\big) \Big\} \xrightarrow{n\to \infty} 0.
    \end{align*} 
\end{theorem}
\noindent Informally, the result states that there is no hole of fixed size at any given point $-t$ in the range of the vector-valued field away from the boundary, since $\max_{x \in \Lambda_n} |\phi(x)| = (2 + o(1)) \log n$ in probability \cite[Theorem 2]{entropRep2D}.

For the norm of the conditioned field, we then prove:
\begin{theorem}[Entropic repulsion of the norm]
\label{thm:repulsion}
    Let $\beta > 0$ and $\varepsilon > 0$. Then, the set of vertices $x \in D_n$ such that 
    \begin{align*}
        \P_n\bigg[\frac{|\phi(x)|}{\log n} \in \big(2 -\beta, 2 + \beta\big) ~\bigg\vert~ \Omega_{D_n}^I\bigg] \leq  1- \varepsilon
    \end{align*}
    is negligible. That is to say, its cardinality divided by $n^{2}$ goes to $0$ as $n\rightarrow \infty$. 
\end{theorem}

Moreover, we describe the angular component $\phi/|\phi|$ of the GFF—referred to as the spin—under this conditioning. In this case, the spins exhibit what we call freezing, or total ordering, at mesoscopic scales as established in the following theorem.
\begin{theorem}[Freezing of spins] \label{thm:freezing}
    Let $\varepsilon>0$, $\nu\in [0,1)$ if $N\geq 2$, and $\nu = 1$ if $N=1$. Then, for all sufficiently large $n$, there exists a set $\tilde D_n \subset D_n$ of cardinality $|D_n|(1-o(1))$ such that for all its vertices $x,y \in \tilde D_n$ with $|x-y| \leq n^\nu$,
    \begin{align}
    \label{eq:freezing_0}
        \E_n \bigg[\frac{\phi(x) }{|\phi(x)| } \cdot \frac{\phi(y) }{|\phi(y)| } ~ \bigg\vert~ \Omega_{D_n}^I \bigg] \geq 1 - \varepsilon.
    \end{align}
\end{theorem}
Additionally, for $N=1$—as in the hard-wall case\footnote{Note that there is a typo in the corresponding result in \cite{entropRep2D}, see \ref{footnote:typo_cond_event}}—we show (cf. Corollary \ref{cor:halfline_implications} and Proposition \ref{prop:N=1upperbound_condition}) that for any $\beta > 0$,
\begin{align*}
    \lim_{n \rightarrow \infty} \frac{\log \P_n[ \Omega_{D_n}^{I}]}{ \log^2 n }  = -\frac{2}{\pi} \:\cpc_\Lambda(D),
\end{align*}
where $\cpc_\Lambda(D)$ is the relative capacity of $D$ with respect to $\Lambda$ defined by 
\begin{align*}
    \cpc_\Lambda(D) =\frac{1}{2} \inf\left \{ \norm{\nabla f}_{L^2(\Lambda)}^2: f \in H^1_0(\Lambda), f \geq 1 \text{ on } D\right \}.
\end{align*} 

Similarly to the hard-wall case, to prove all the above results we study the field at mesoscopic scales $n^\alpha$ for appropriate values of $\alpha \in (0,1)$, where the local spikes are living. More precisely, we divide $D_n$ into boxes of side length $n^\alpha$ such that this subgrid is centered at some vertex $x_0 \in \Lambda_{n^\alpha}$—denote this collection of boxes by $\Pi_\alpha(x_0)$.

Theorem \ref{intro:thm:2Dno_hole} follows from three components: bounds obtained in the proof of \cite[Theorem 1.3]{extremesGFF} on the number of $\eta$-high points of the GFF (adapted to the vector-valued case), the domain Markov property of the GFF (Lemma \ref{lemma:domainMP}), and the fact that the number of boxes $B \in \Pi_\alpha(x_0)$ contained in $\B^2(0, r_n n)$, where at the center of $B$, all the coordinates of the harmonic extension $h_B$ of $\phi |_{\partial_{\text{in}} B}$ are positive, cannot be of smaller order than the total number of such boxes in $\Pi_\alpha(x_0)$ (Lemma \ref{lemma:pos_centers_estimate}).

As for the entropic repulsion, the upper bound for the norm of the conditioned field is a simple consequence of the hard-wall case. The lower bound follows easily from two key observations (proven in Subsection \ref{subsec:key_ingredients_repulsion}): on the event $\Omega^I_{D_n}$, for a fixed small $\beta > 0$
\begin{enumerate}
    \item The number of boxes $B \in \Pi_\alpha(x_0)$, where at the center of $B$, the harmonic extension $h_B$ has the norm smaller than $(2 - \beta/2) \log n$, is negligible (uniformly in $x_0$) if $\alpha$ is close enough to one;
    \item The number of boxes $B \in \Pi_\alpha(x_0)$, for which $|\phi_{x_B} - h_B|$ exceeds $\frac{\beta}{2} \log n$ is negligible (uniformly in $x_0$).
\end{enumerate}
These two ingredients—with a slightly generalized version of the second point (see Proposition~\ref{prop:few_high_fluct_boxes})—also imply, with a suitable choice of $\tilde D_n$, freezing at all mesoscopic scales $n^\nu$ for $\nu \in [0,1)$.

The first ingredient relies primarily on the fact that the probability of a hole appearing in the range of the GFF is small (cf. Theorem \ref{intro:thm:2Dno_hole}), while the second follows from the domain Markov property of the GFF and Bernstein's inequality for independent bounded random variables. 

To prove freezing of signs at macroscopic scales in the case $N = 1$, in addition to the two ingredients above, we need one further input (proved in Subsection \ref{subsubsec:pick_sign}):
\begin{enumerate}
\setcounter{enumi}{2}
    \item Given $\Omega^I_{D_n}$ and $\alpha \in (0,1)$, the sign of $h_B$ remains constant in almost all boxes of $\Pi_\alpha(x_0)$ (for all $x_0 \in \Lambda_{n^\alpha}$).
\end{enumerate}
Combined with the two previous ingredients, it implies that the set $G_\alpha(x_0)$ of centers of boxes $\Pi_\alpha(x_0)$, for each $x_0 \in \Lambda_{n^\alpha}$, picks a sign—in the sense that $\phi(x)$ at ``almost all" centers is either positive or negative. Furthermore, the union of the disjoint sets $\Cc^{\sign}_\alpha(z, \beta, \varepsilon)$ defined as
\begin{align*}
    \bigg\{x \in G_\alpha(x_0): \P_n \bigg[ \frac{|\phi(x)|}{\log n} \geq 2 - \beta, \,\sgn(\phi(x)) = \sign(G_\alpha(x_0))~\bigg|~\Omega^I_{D_n}\bigg] \geq 1 - \varepsilon \bigg\}
\end{align*}
contains all but $o(n^2)$ points of $D_n$. Since for two grids $G_\alpha(x), G_\alpha(y)$ at distance $n^\alpha/K$, for $K>0$ sufficiently large, it is highly unlikely that they pick different signs, the vertices in the above union satisfy \eqref{eq:freezing_0}. 

The idea behind the third ingredient is that if the number of both positive ($h_B > 0$) and negative ($h_B < 0$) boxes were considerable, then at the interface between the negative and positive ones, there would either be considerably many \emph{low} boxes—those with $|h_B| < \delta \log n$ for some small $\delta > 0$—or jumps of magnitude $\delta \log n$, i.e., values of $|h_B - h_{B'}|$ for neighboring boxes $B$ and $B'$ of opposite sign. Both of these events are extremely unlikely.

\subsection{Massive case}
For $m^2>0$, let $\psi \sim \P_{\Z^2, m^2}$ be an $N$-vectorial $m^2$-massive Gaussian free field on $\Z^2$. As before, for an open ball $I = \B^N(0,R) \subset \R^N$ and $V \subset \Lambda_n$, let 
$\Omega_{V}^I \coloneqq \{\psi(x) \notin I, \; \forall x \in V\}$. Denote by $\G_1$ the set of infinite-volume Gibbs measures $\mu$ heuristically corresponding to the law of $\psi$ conditioned to avoid $I$ at every site of $\Z^2$ (``$\P_{\Z^2, m^2}[ \, \cdot\, |\Omega^I_{\Z^2}]$'') and possessing uniform first moments at each site; that is, $\mu$ satisfies $\sup_{x \in \Z^2} \mu[|\psi(x)|] \leq C < \infty$.
\begin{theorem}[Massive GFF conditioned to avoid a ball]
\label{thm:massiveGFF}
    The set $\G_1$ is non-empty. Furthermore, if $R>0$ is sufficiently small (depending on $m^2, N, d$, under an additional assumption on $m^2$), $|\G_1| = 1$. On the other hand, for $N=1$, the model undergoes a phase transition: for sufficiently large $R>0$, $|\G_1| > 1$. 
\end{theorem}
\noindent See Propositions \ref{prop:massiveGFF_existence} and \ref{prop:massiveGFF_uniqueness} for the precise statements. 

To prove the existence of the desired Gibbs measures, we consider the sequence of probability measures $(\P_{\Z^2, m^2}[\cdot| \Omega^I_{\Lambda_n}])_n$ and show that it is tight, with uniformly bounded first moment at any given site. The latter follows from the FKG inequality for the conditioned field (Proposition \ref{prop:FKG_condGFF}) and standard entropy bounds. Uniqueness for sufficiently small $R>0$ is then established in a canonical way, by verifying Dobrushin's uniqueness criterion \cite[Theorem 4]{Dobrushin}. Finally, in the scalar case with $R$ sufficiently large, we consider a similar tight sequence of probability measures $(\tilde \Q_n)_n$, where we condition on $\Omega^I_{\Lambda_n}\cap \Omega^{(-\infty, R)}_{\Lambda_{2n}\setminus \Lambda_n}$ instead of $\Omega_{\Lambda_n}^I$, and show that any of its subsequential limits differs from the one from above. This is done by comparing the law of $\psi(0)/|\psi(0)|$ given $(|\psi(x)|)_{x \in \Lambda_n}$ under $\tilde \Q_n$ with the Ising model with plus boundary condition at inverse temperature $R^2$.

\subsection*{Outline}
Section \ref{sec:setup} defines and provides the necessary background on the discrete (vector-valued massive) GFF and its covariance structure. Section \ref{sec:2D_GFF_range} discusses the range of the vector-valued GFF and proves that there is no hole of fixed size at any given point in the bulk of the range, which is then used in the proof of the entropic repulsion. In Sections \ref{sec:repulsion} and \ref{sec:freezing}, we study a vector-valued Dirichlet GFF conditioned to avoid a ball: the former section proves the entropic repulsion phenomenon (Theorem \ref{thm:repulsion}), while the latter demonstrates the freezing of spins of the conditioned field (Theorem \ref{thm:freezing}). In Section \ref{sec:massiveGFF}, we treat the massive vector-valued lattice field under an analogous condition—avoiding a ball in a box of side length~$n$—and prove Theorem \ref{thm:massiveGFF}.

\subsection*{Acknowledgments}
We thank Juhan Aru for insightful discussions.
Part of this work was carried out during the research visit of A.K. as a participant of the CMM-PhD Visiting Program 2024 to the Center for Mathematical Modeling at the University of Chile. 
The research of A.K. was supported by Eccellenza grant 194648 of the Swiss National Science Foundation and is a members of NCCR Swissmap.
The research of A.S. was supported by Centro de Modelamiento Matem\'{a}tico Basal Funds FB210005 from ANID-Chile, by Fondecyt Grant 1240884, and by ERC 101043450 Vortex. 
A.S. would also like to thank the Hausdorff Institute of Mathematics, and in particular the trimester program ``Probabilistic methods in quantum field theory'', where he was based while part of this work was carried out.


\section{Setup and preliminaries}
\label{sec:setup}

\subsection{Notation and definitions}
\label{subsec:notation&defs}
We begin this section by defining the the main object of study of this paper, the vector-valued (massive Dirichlet/zero-boundary) Gaussian free field on a two-dimensional graph. To this end, let $G \subset \Z^2$ be a given finite graph, and let $\partial G$ stand for the outer boundary\footnote{Here, by the outer boundary we mean the vertices in $\Z^2\setminus G$ that have a nearest neighbor in $G$.} of its set of vertices in $\Z^2$. From now on, we furthermore denote the norm of a vector $x\in \R^k$ for any $k \in \N$ by $|x|$.
 
\begin{definition}[($m^2$-massive $N$-vectorial) Gaussian free field]
\label{def:GFF}
    Let $m \geq 0$, $N \in \N$, $\mathfrak{g} > 0$, and $G \subset \Z^2$. We call a random function $\phi: G \rightarrow \R^N$ an $m^2$-massive $N$-vectorial Gaussian free field on $G$ (with Dirichlet boundary condition on $\partial G$) if its law is given by
    \begin{align*} 
        \P[\d \phi] \propto \exp \bigg(- \frac{\mathfrak{g}}{4}\sum_{x\sim y} |\phi(x) - \phi(y)|^2 - \frac{m^2}{2} \sum_{x \in G} |\phi(x)|^2\bigg) \prod_{x \in G} \d \phi(x),
    \end{align*}
    where here and in the sequel sums over $x\sim y$ run over the pairs $(x, y)$ such that $\{x,y\}$ is an edge in $G \cup \partial G$. We write $\phi(x)$ to mean $0$ whenever $x \in \partial G$. If $m^2 = 0$, we drop the term ``$m^2$-massive'' from the name.
\end{definition}
\begin{remark}[Normalization of the GFF]\label{rem:normalization}
    In the massless case ($m^2 = 0$), we assume throughout that $\mathfrak{g} = \frac{1}{2\pi}$, so that the variance of a coordinate field at the origin satisfies $\E[(\phi^1(0))^2] \sim \log n$. This choice considerably shortens many expressions in the statements of results and proofs. Note, however, that this is not the standard normalization of the discrete field; whenever we use results obtained for a different normalization of the field, we state them already adjusted to our setting. In the massive case ($m^2 > 0$), treated in Section \ref{sec:massiveGFF}, we instead take $\mathfrak{g} = 1$ to simplify the expressions there.
\end{remark}

One of the main technical tools in our work is the domain Markov property of the GFF. This property will be used to obtain independent behavior on the field, despite the field itself not decorrelating.
\begin{lemma}[Domain Markov property]
\label{lemma:domainMP}
     Let $G\subseteq \Z^2$ be a finite graph, $K\subseteq G$ and $m\geq 0$. For $\phi$ an $m^2$-massive GFF on $G$, set
     \begin{align*}
        h^{K}(x) = \E[\phi(x) |\F_K] \quad \text{for all } x \in G,
    \end{align*}
    where $\F_K \coloneqq \sigma(\phi(x): x \in K)$. Then, $h^{K}$ is discrete (massive-) harmonic\footnote{That is to say $(-\Delta_{G, \mathfrak{g}} + m^2) h^{K} (y) = (4\mathfrak{g} + m^2) h^K(y) - \mathfrak{g} \sum_{z \in G: z\sim y} h^K(z) = 0$ for any $y \in G\setminus K$.} on $G\setminus K$ with boundary values determined by $h^K(x) = \phi(x)$ for any $x \in \partial G \cup K$. Moreover, the field $\phi^{G\setminus K} \coloneqq \phi-h^K$ is independent of $\F_K$ and has the law of an $m^2$-massive GFF on $G\setminus K$ (with values $0$ on $K$ if viewed as a function on $G$). 
\end{lemma}
\begin{proof}
    See \cite[Lemma 3.1]{Biskup_notes} and \cite[Lemma 1.1]{Rod17} for the massless and massive cases, respectively. 
\end{proof}

We now introduce the setup of the paper. We take $\Lambda = [-1/2,1/2]^2$ and $D\subseteq \Lambda$ an open set with smooth boundary such that $0 \in D$ and $\upiota:= \mathrm{dist}(D, \partial \Lambda)>0$. For a given $n\in \N$, we consider a discrete blow-up version of $D$ defined as $D_n = (nD) \cap \Z^2$ and the domain $\Lambda_n = [-n/2, n/2]^2 \cap \Z^2$. Let $\alpha \in (0,1)$, $x_0 \in \Z^2$. Abusing notation and writing $n^\alpha$ for the closest even integer to $n^\alpha$, we define a mesoscopic discretization $\Pi_\alpha(x_0)$ of $D_n$ at scale $n^\alpha$ as the set of boxes $B \subset D_n$ of side length $n^\alpha$ with centers in $x_0 + n^\alpha \Z^2$ (see Figure \ref{fig:setup}). We denote the union of the inner boundaries of these boxes $B$ by $\primorial_\alpha(x_0)$, and their centers by $x_B$. For brevity, we write $\Pi_\alpha(x_0) \cap F$ for the subset of boxes $B \in \Pi_\alpha(x_0)$ contained in $F \subset \Z^2$, i.e., $B \subset F$. 
\begin{figure}[ht]
    \centering
    \includegraphics[width=0.7\linewidth]{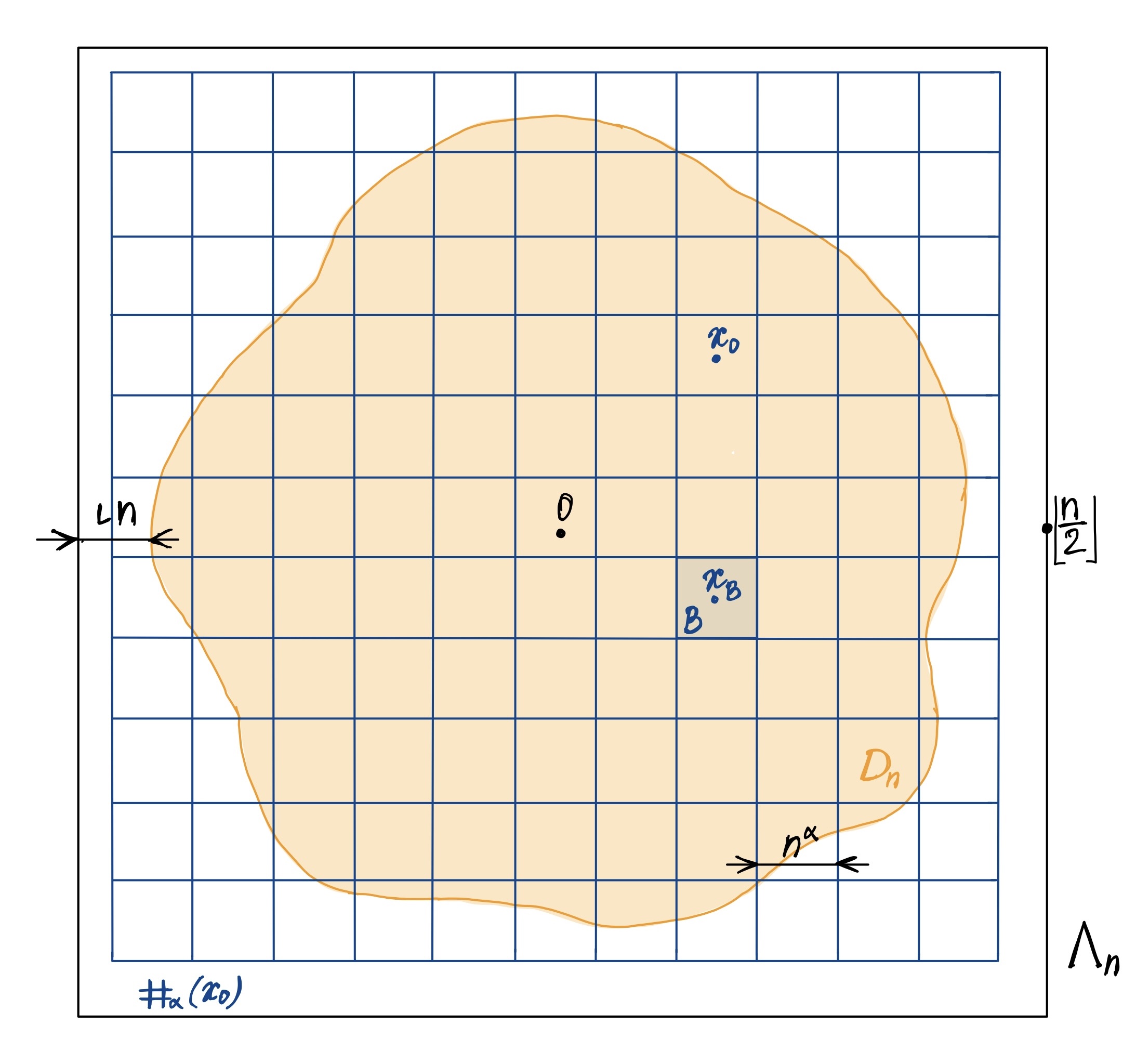}
     \caption{The part of the \emph{blue} subgrid of $\Z^2$ fully contained in $D_n$ is $\primorial_\alpha(x_0)$. The union of boxes $B$ formed by this part of the subgrid is $\Pi_\alpha(x_0)$.}
    \label{fig:setup}
\end{figure} 

Unless stated otherwise, we write $\phi \sim \P_n$ for an ($N$-vectorial) Dirichlet GFF on $\Lambda_n$. Then—as in the domain Markov property (Lemma \ref{lemma:domainMP})—$h^{\primorial_\alpha(x_0)}$ stands for the harmonic extension of its values on $\primorial_{\alpha}(x_0)$, and for better readability, we set $h_B \coloneqq h^{\primorial_\alpha(x_0)}_{x_B}$. Furthermore, for every $B \in \Pi_\alpha(x_0)$, define $\phi^B(x)$ by $(\phi^{\Lambda_n \setminus \primorial_\alpha(x_0)}(x)) \ind_{\{x \in B\}}$. As the direct consequence of the domain Markov property, we get that the family $(\phi^B)_{B \in \Pi_\alpha(x_0)}$ is independent of $\F_\alpha(x_0) \coloneqq \sigma(\phi_x: x \in \,\primorial_\alpha(x_0))$, and has the law of a collection of jointly independent Dirichlet GFFs on the interiors of $B$'s. If $x_0 = 0$, for brevity, we suppress the dependence on $x_0$ in all the above expressions.

Furthermore, we write $\psi$ for an $m^2$-massive field on $\Z^d$ ($d\geq 2$, cf. \cite[Section 8.5]{StatMech_book}), and denote its law by $\P_{\Z^d, m^2}$. Sometimes to highlight the different spin-dimensions used within a sequence of equations we add $N$ to the subscript of $\P_n$, $\P_{\Z^d, m^2}$. More precisely, in the equations involving, say, both $\P_n$ and $\P_{n,N}$ the former stands for the law of the scalar field, while the latter denotes the law of the $N$-vectorial field.

\subsection{Correlations and FKG inequality}
The goal of this section is to recall and prove some technical results related to the (scalar) Dirichlet GFF $\phi$ on $\Lambda_n$. In particular, we first recall an FKG inequality for the GFF conditioned to avoid specific sets and then study the correlation structure of $h^K$.

To respect the above setup for the massless field, we state the FKG inequality for the planar case. However, note that with obvious adjustments to the statement, it holds (with the same proof) for any $d\geq 2$, which we then use in Section \ref{sec:massiveGFF} in our study of the massive GFF on $\Z^d$.
\begin{proposition} \label{prop:FKG_condGFF}
    Let $V$ be a finite subset of $\Lambda_n$ and $(A_x)_{x \in V}$ be a family of subsets of $\R$ consisting of finite unions of possibly infinite intervals. Let $\phi \sim \P$ be either a massless Dirichlet GFF on $\Lambda_n$, or an $m^2$-massive GFF on $\Z^2$ or $\Lambda_n$.
    Then, $\P$ conditioned on the event $\Omega^{A^c}_{V} \coloneqq \{\phi(x) \in A_x, \; \forall x \in V\}$ satisfies the FKG inequality: for any increasing functions $f,g$
    \begin{equation}
        \label{eq:FKG}
        \E \big[f(\phi) g(\phi)~\big\vert~\Omega^{A^c}_V\big] \geq \E \big[ f(\phi) ~\big\vert~\Omega^{A^c}_V\big]\, \E \big[ g(\phi) ~\big\vert~ \Omega^{A^c}_V\big].
    \end{equation}
\end{proposition}
\begin{proof}
    The result can be easily deduced from the proofs of \cite[Lemma 1.3]{Rod17} (to reduce \eqref{eq:FKG} for $\P_{\Z^d, m^2}$ to an analogous statement for the $m^2$-massive Dirichlet GFF on $\Lambda_n$) and \cite[Lemma 2.6]{AGS} (adding a mass-term to the Hamiltonian does not break the Holley condition).
\end{proof}

We now turn to the covariance structure $G_{\Lambda_n}$ of $\phi$, the so-called Dirichlet Green's function\footnote{We remain consistent with our chosen normalization of the GFF, see Remark \ref{rem:normalization}}.
\begin{lemma}[Estimates for the planar Dirichlet Green's function]
\label{lemma:planar_Greens_fct}
    For all $x,y \in \Z^2$, let $G_0(x,y) \coloneqq -\frac{\pi}{2} a(y-x)$ with $a$ being the potential kernel of a simple random walk on $\Z^2$ (cf. \cite[Section 1.6]{lawler2012intersections}). Then, for $x = (x_1, x_2) \in \Z^2$ as $|x|\rightarrow \infty$,
    \begin{align} \label{eq:potkernel_expansion}
        G_0(0,x) = -\log |x| + C_1 + C_2 \frac{x_1^2 x_2^2}{|x|^6} + C_3\frac{1}{|x|^2} + \mathcal{O}(|x|^{-3})
    \end{align}
    for some constants $C_1, C_2, C_3 \in \R$.

    Let $\Lambda_n$ and $D_n$ be as above—in particular, $D_n$ is at distance at least $\upiota n/2$ from $\partial \Lambda_n$ for some $\upiota >0$. Then, there exists $C = C(\upiota) >0$ such that for all $x, y \in D_n$, $y\neq 0$, 
    \begin{equation}
    \begin{aligned}\label{eq:Greensfct_bounds}
        |G_{\Lambda_{n}}(x,x) - \log n|\vee |G_{\Lambda_n} (0,y) - \log (n /|y|)| &\leq C.
    \end{aligned}
    \end{equation}
    Furthermore, there exist $c = c(\upiota) > 0$ and $n_0 = n_0(\upiota) \in \N$ such that for all $n \geq n_0$, for all $x,y \in D_n$,
    \begin{align} \label{eq:Greensdiff_bounds}
        \big|\partial^{(1)}_{e_i}\partial^{(2)}_{e_j} G_{\Lambda_n}(x,y) - \partial^{(1)}_{e_i}\partial^{(2)}_{e_j} G_{0}(x,y)\big| \leq \frac{c(\upiota)}{n^2},
    \end{align}
    where $e_1 = 1, e_2 = \iu \in \Z^2 \subset \C$ and $\partial^{(k)}_{e_l}$ denotes the discrete $e_l$-directional derivative in the $k$-th component: for example, $\partial^{(2)}_{e_j} f(u,v) = f(u, v+e_j) - f (u,v)$. 
\end{lemma}
\begin{proof}
    \eqref{eq:potkernel_expansion} and \eqref{eq:Greensdiff_bounds} are \cite[Lemma 28 and Lemma 31]{duerre_thesis}, respectively. For \eqref{eq:Greensfct_bounds}, see \cite[Theorem 1.6.6 and Proposition 1.6.7]{lawler2012intersections}. 
\end{proof}

Next, we discuss the covariance structure of $h^K$. Observe that, as an immediate consequence of the domain Markov property of the GFF (Lemma \ref{lemma:domainMP}), we have
\begin{align} \label{eq:cov_harmext}
    \E_n[h^{K}(x) h^{K}(y)] &= G_{\Lambda_n}(x,y) - G_{\Lambda_n \setminus K}(x,y),
\end{align}
where $G_{\Lambda_n \setminus K}(x,y)$ is the Green's function on $\Lambda_n$ with Dirichlet boundary condition on $K$. We use this identity to study the behavior of $h^{\primorial_\alpha(x_0)}$, for $\alpha \in (0,1)$ and $x_0 \in \Z^2$, by setting $K \coloneqq \, \primorial_\alpha(x_0)$ throughout the rest of this subsection. 

We start by showing that $h^K$ does not vary much close to the centers of the boxes in $\Pi_\alpha(x_0)$:
\begin{lemma}
    There exist constants $\tilde c = \tilde c(\upiota) > 0$ and $n_0 = n_0(\upiota, \alpha) \in \N$ such that for all $n \geq n_0$ and $0 \leq r < 1/4$ (may depend on $n$, e.g., $r = o(1)$ as $n\to \infty$ is included),
    \begin{align}\label{eq:harmdiff_bound}
        \sup \Big\{ \Var_n\big[h^K(x) - h_{B}\big]: B \in \Pi_\alpha(x_0); \;x \in B,\, |x - x_B| \leq r n^\alpha\Big\} \leq \tilde c r^2.
    \end{align}
\end{lemma}
\begin{proof}
    Note that for a function $f: V^2 \subset (\Z^2)^2 \rightarrow \R$ satisfying $f(u,v) = f(v,u)$ and $k, k' \in \N_0, v \in V$ such that $v + k + k'\iu \in V$, 
    \begin{align*}
        \Lc_{k,k'} f (v, v) &\coloneqq f(v + k + k'\iu, v+k + k'\iu) + f(v,v) - 2f(v, v + k + k'\iu)\\
        &\:= \sum_{m,l = 0}^{k-1} \partial^{(1)}_{e_1} \partial^{(2)}_{e_1} f(v + m + k' \iu, v + l + k'\iu) + \sum_{m',l' = 0}^{k'-1} \partial^{(1)}_{e_2} \partial^{(2)}_{e_2} f(v + m' \iu, v + l'\iu) \\
        &\:+ 2\sum_{l = 0}^{k-1} \sum_{m'=0}^{k'-1} \partial^{(1)}_{e_2} \partial^{(2)}_{e_1} f(v + m' \iu, v + l + k'\iu).
    \end{align*}
    Let $x, x_B \in B \subset D_n$ as in \eqref{eq:harmdiff_bound}—note that, by symmetry, we may assume without loss of generality that $x_1 \geq (x_B)_1$ and $x_2\geq (x_B)_2$. Then, setting $v = x_B$, $k = x_1 - (x_B)_1$, and $k' = x_2 - (x_B)_2$ (so in particular $0 \leq k, k' \leq r n^\alpha$), the above together with \eqref{eq:Greensdiff_bounds} implies
    \begin{align*}
        \Var_n[h^K(x) - h_{B}] &= \Lc_{k,k'} G_{\Lambda_n}(v,v) - \Lc_{k,k'} G_B(v,v) \\
        &\leq (k^2 + (k')^2 + 2 k k') \Big(\frac{c(\upiota)}{n^2} + \frac{c(1/2)}{n^{2\alpha}}\Big) \leq \tilde c r^2
    \end{align*}
    for all $n$ sufficiently large as desired. 
\end{proof}

We are also interested in the covariance matrix of the process of differences of the form $h_B - h_{B'}$ at centers of the neighboring boxes of $\Pi_\alpha(x_0)$ and prove the following result about its diagonal values and eigenvalues.
\begin{lemma}
\label{lemma:harmonic_cov_EVs}
    The covariance matrix $\Sigma$ of the Gaussian process $(Z({B, p}))_{B \in \Pi_\alpha(x_0), p \in \{1, \iu\}}$ defined by $Z({B, p}) = h_{B + n^\alpha p} - h_{B}$ satisfies the following. There exists $\tilde c = \tilde c(\upiota) > 0$\footnote{Note that by definition of $\Pi_\alpha(x_0)$, all the centers $x_B$ as well as $x_B + n^\alpha p$ for $p=1, \iu$ are at distance at least $\upiota n/4$ from $\partial \Lambda_n$ for all $n$ sufficiently large depending on $\alpha$ and $\iota$.} such that for all $n$ sufficiently large (depending on $\upiota, \alpha$)
    \begin{enumerate}
        \item ${\Sigma}_{(B,p), (B, p)} \leq \tilde c$ for all $B \in \Pi_\alpha(x_0)$ and $p = 1, \iu \in \Z^2 \subset \C$;
        \item The maximal eigenvalue of ${\Sigma}$ is bounded by $\tilde c$.\footnote{Since $\Sigma$ is positive semi-definite, the largest eigenvalue of any principal submatrix is also bounded by $\tilde{c}$.} 
    \end{enumerate}
\end{lemma}
\begin{proof}
    Recall that $\E_n[h^K(x) h^K(y)] = G_{\Lambda_n}(x, y) - G_{\Lambda_n \setminus K}(x,y)$, implying that $\Sigma_{(B,p), (B, p)}$ equals
    \begin{align*}
         G_{\Lambda_n} (x_B + n^\alpha p, x_B + n^\alpha p) + G_{\Lambda_n}(x_B, x_B) - 2 G_{\Lambda_n}(x_B, x_B + n^\alpha p) - 2 G_{B_0}(0,0),
    \end{align*}
    which by \eqref{eq:Greensfct_bounds} is upper-bounded by an appropriate constant $C>0$ depending only on $\upiota$. Here, $B_0$ denotes the box of side length $n^\alpha$ centered at $0 \in \Z^2$.
    
    To prove the second statement, we compare $\Sigma$ to the covariance matrix of the (discrete) gradients of a Dirichlet GFF $\tilde \phi$ on $\Lambda_{n^{1-\alpha}}$, for which an analogous statement follows from \cite[Proposition 2.7]{GS_diff_calc}. 
    More precisely, let $G$ be $\Lambda_{n^{1-\alpha}}$ together with its outer boundary viewed as a planar graph, and $E(G)$ be the set of its undirected edges. Consider the process $\tilde Z({y, p})\coloneqq \tilde \phi(y+p) - \tilde\phi(y)$ for $y \in \Lambda_{n^{1-\alpha}}, p \in \{1, \iu\} = \{e_1, e_2\}$ such that $\{y, y+p\} \in E(G)$, and denote its covariance matrix by $\Xi$. 
    The aforementioned result \cite[Proposition 2.7]{GS_diff_calc} yields the existence of an independent centered Gaussian process $Z'$ such that the sum $Z + Z'$ has the law of $\sigma$ times a standard normal vector indexed by $E(G)$, for some absolute constant $\sigma > 0$. This in turn, implies that the maximal eigenvalue of $\Xi$ is bounded by $\sigma^2$ (since the covariance matrix of the added field $\sigma^2\mathrm{Id} - \Xi$ has to be positive semi-definite). Hence, the largest eigenvalue of any of its principal submatrices—as $\Xi$ is positive semi-definite—is also bounded by $\sigma^2$. 
    
    Observe next that $\Xi_{(x,p), (y, q)} = \partial^{(1)}_{p} \partial^{(2)}_q G_{\Lambda_{n^{1-\alpha}}}(x, y)$, which combined with \eqref{eq:Greensdiff_bounds} shows that 
    \begin{align}
    \label{eq:cov_gradfield}
        \big|\Xi_{(x,p), (y, q)} - \partial^{(1)}_{p} \partial^{(2)}_q G_0(x, y)\big| \leq \frac{c(\upiota/2)}{n^{2(1-\alpha)}} 
    \end{align}
    for all $x, y, x+p, y+q \in \Lambda_{n^{1-\alpha}}$ ($p,q \in \{1,\iu\}$) at distance at least $\upiota n^{1-\alpha}/4$ from the outer boundary of $\Lambda_{n^{1-\alpha}}$, as long as $n$ is sufficiently large, depending on $\upiota, \alpha$. 
    
    On the other hand, similarly to the above, for any $B, B' \in \Pi_\alpha(x_0)$ with $|x_B - x_{B'}| \geq 2n^\alpha$ (so that the four boxes $B, B+n^\alpha p, B'$, and $B'+n^\alpha q$ are disjoint up to the boundaries),
    \begin{equation*}
    \begin{aligned}
        \Sigma_{(B, p), (B', q)} &= G_{\Lambda_n}(x_B + n^\alpha p, x_{B'} + n^\alpha q) + G_{\Lambda_n}(x_B, x_{B'}) \\
        &\qquad \qquad- G_{\Lambda_n}(x_B + n^\alpha p, x_{B'}) - G_{\Lambda_n}(x_B, x_{B'} + n^\alpha q)\\
        &=\sum_{k,l=0}^{n^\alpha - 1} \partial_p^{(1)} \partial_q^{(2)} G_{\Lambda_n}(x_B + k p, x_{B'} + lq).
    \end{aligned}
    \end{equation*}
    And thus, by \eqref{eq:Greensdiff_bounds}, $\Sigma_{(B, p), (B', q)}$ belongs to the interval $\frac{c(\upiota/2)}{n^{2-2\alpha}}(-1, 1)$ shifted by
    \begin{align}\label{eq:gradharm_centers}
        G_{0}(x_B + n^\alpha p, x_{B'} + n^\alpha q) + G_{0}(x_B, x_{B'})- G_{0}(x_B + n^\alpha p, x_{B'}) - G_{0}(x_B, x_{B'} + n^\alpha q). 
    \end{align}

    The goal is now to compare the shift in \eqref{eq:gradharm_centers} to $\partial^{(1)}_{p} \partial^{(2)}_q G_0(x, y)$ in \eqref{eq:cov_gradfield} for appropriate $x,y$. To this end, set $x \coloneqq \frac{x_B - \hat{x}_0}{n^\alpha}$ and $y \coloneqq \frac{x_{B'}-\hat{x}_0}{n^\alpha}$, where $\hat{x}_0$ is the unique element in $(x_0 + \Z^2) \cap \Lambda_{n^\alpha}$. With this choice, $x, y, x+p, x+q \in \Lambda_{n^{1-\alpha}}$ are clearly at distance at least $\upiota n^{1-\alpha}/4$ from the outer boundary of $\Lambda_{n^{1-\alpha}}$. Thus, by \eqref{eq:potkernel_expansion},
    \begin{align*}
        G_{0}(x_B &+ n^\alpha p, x_{B'} + n^\alpha q) + G_{0}(x_B, x_{B'}) - G_{0}(x_B + n^\alpha p, x_{B'}) - G_{0}(x_B, x_{B'} + n^\alpha q) \\
        &= \frac{1}{4} \big(a(y+ q - x) + a(y- x - p) - a(y - x) - a(y + q - x -p)\big) + \mathcal{O}(|x-y|^{-3})\\
        &= \partial_p^{(1)}\partial_q^{(2)}G_0(x,y) + \mathcal{O}(|x-y|^{-3})
    \end{align*}
    as $|x-y|$ (together with $n$) tends to infinity. This allows us to conclude that there exists $\hat c = \hat c(\upiota) > 0$ such that for all $B, B' \in \Pi_\alpha(x_0), p, q \in \{1, \iu\}$ and $x, y$ (defined in terms of $x_B$, $x_{B'}$ and $\hat{x}_0$) as above,
    \begin{align}\label{eq:cov_diff_comaparison}
        \big|\Sigma_{(B, p), (B', q)} - \Xi_{(x,p), (y, q)}\big| \leq \hat{c} \bigg(\frac{1}{n^{2-2\alpha}} + \frac{1}{|x-y|^3 \vee 1} \bigg)
    \end{align}
    for all $n$ sufficiently large (depending on $\alpha$ and $\upiota$). 
     
    Finally, if $U$ denotes the set of vertices $\big\{\frac{x_B - \hat{x}_0}{n^\alpha}: B \in \Pi_\alpha(\hat{x}_0)\big\}$, \eqref{eq:cov_diff_comaparison} and Cauchy-Schwarz inequality imply, 
    \begin{align*}
        \lambda_{\max}(\Sigma) &= \sup_{v\in \R^{\Pi_\alpha(x_0)}: \norm{v} = 1} \langle v, \Sigma v\rangle \\
        &\leq \sup_{v\in \R^{\Pi_\alpha(x_0)}: \norm{v} = 1} \langle v, \Xi\vert_{U \times U} v\rangle 
        + \hat{c} \bigg(\frac{|\Pi_\alpha(x_0)|}{n^{2-2\alpha}} + \sup_{x \in U} \sum_{y\in U} \frac{1}{1 \vee |x-y|^3}\bigg) \\
        &\leq \lambda_{\max}(\Xi) + \hat{c}K \leq \sigma^2 + \hat{c}K,
    \end{align*}
    where $K = 2 + \sum_{y \in \Z^2\setminus\{0\}} |y|^{-3} < \infty$.
\end{proof}


\section{Range of the planar vector-valued Gaussian free field} 
\label{sec:2D_GFF_range}

In this section, we study the range of the $N$-vectorial planar Dirichlet Gaussian free field $\phi\sim \P_n$ on $\Lambda_n$ and prove that there is no hole of fixed size at any given point $-t$ in the bulk of the range of the vector-valued field:
\begin{theorem} \label{thm:2Dno_hole}
    Let $\varepsilon > 0$, $I$ be a bounded subset of $\R^N$ with non-empty interior, and $r_n > 0$ be such that $\frac{\log (1/r_n)}{\log n} = o(1)$. Then,
    \begin{align*} 
        \sup \Big\{\P_n\Big[\phi(x) + t \notin I, \;\forall x \in \Lambda_n \cap \mathcal{B}^2(0, r_n n)\Big]: t\in \B^N \big(0,  2(1-\varepsilon)\log n\big) \Big\} \xrightarrow{n\to \infty} 0.
    \end{align*} 
    Here and in the sequel, $\B^k(0,s) \subset \R^k$ denotes the centered $k$-dimensional ball of radius $s$.
\end{theorem}
\begin{remark}
    \label{rem:max_2D_vectorial}
    In the scalar case, it is known that the supremum of $\phi$ over $\Lambda_n$ equals $(2 + o(1)) \log n$, where $o(1) \to 0$ as $n\to \infty$ in probability (cf. \cite[Theorem 2]{entropRep2D}). This result straightforwardly translates to the vector-valued case in the sense that $\sup_{x\in \Lambda_n} |\phi(x)| = (2 + o(1)) \log n$ in probability. The lower bound for the supremum follows from the fact that $|\phi(x)|\geq \phi^1(x)$, while the upper bound is a consequence of the union bound, \eqref{eq:Greensfct_bounds}, and classical concentration inequalities (for example, Laurent-Massart bounds) for a standard normal vector.   
\end{remark}

Observe that in Theorem \ref{thm:2Dno_hole}, one can without loss of generality restrict to $r_n \leq 1/4$ such that $\B^2(0, r_n n) \cap \Z^2 \subset D_n$. Under this additional assumption, we work in the remainder of this section.

The proof of Theorem \ref{thm:2Dno_hole} is based on analyzing the harmonic extension of the values of the vector-valued GFF on $\primorial_\alpha$, and in particular on estimating the number of boxes for which $h_B/\log n$ is close to a given value. More precisely, for $s \in \B^N(0, 2)$ and $\delta >0$, we define
\begin{align*}
    N^\alpha_{s-\delta, s+\delta} &\coloneqq \#\bigg\{B \in \Pi_\alpha \cap \B: s-\delta< \frac{h_B}{\log n} < s + \delta\bigg\}.
\end{align*}
Here and throughout this section, for brevity, we set $\B \coloneqq \B^2(0, r_n n)$. Moreover, for any $k\in \N$, vectors $v_1,v_2 \in \R^k$, and $u \in \R$, we write $v_1 + u < v_2$ to mean the component-wise relation $v_1^i + u < v_2^i$ for all $i\in \{1,..,k\}$.

The key input for the proof of Theorem \ref{thm:2Dno_hole} can now be formulated as follows.
\begin{proposition}\label{prop:no_hole_harmonic}
    For any $\varepsilon > 0$, $s\in \R^N$ with $s^1\geq \dots \geq s^N \geq 0$ and $|s| \leq 2(1-\varepsilon)$, there exist parameters $\alpha\in (0,1)$, $\kappa \in (0, 2(1-\alpha))$, and $\delta \in (0,\sqrt{2\alpha \kappa/N})$ such that 
    \begin{align*}
        \lim_{n\to \infty} \P_{n}\big[ N^{\alpha}_{s-\delta, s+\delta} \leq n^\kappa\big] =0. 
    \end{align*}
\end{proposition}

We first explain how this proposition implies Theorem \ref{thm:2Dno_hole}.
\begin{proof}[Proof of Theorem \ref{thm:2Dno_hole} assuming Proposition \ref{prop:no_hole_harmonic}]
Our goal is to reduce the problem to studying $s_* \log n$ at a suitable deterministic point $s_*$, instead of taking the supremum over arbitrary $t \in \B^N(0, 2(1-\varepsilon)\log n)$. Note that since $\tilde \phi \coloneqq (\sgn(t^{\sigma(i)}) \phi^{\sigma(i)})_{i=1}^N$, where $\sigma$ is a deterministic permutation of ${1,\dots,N}$ and $\sgn(u) = \ind_{\{u \geq 0\}} - \ind_{\{u < 0\}} \in \{-1,1\}$, has the same law as $\phi$, it suffices to consider $t \in \B^N(0, 2(1-\varepsilon)\log n)$ with $t^1\leq \dots \leq t^N\leq 0$. We then define
    \begin{align*}
        s^i \coloneqq \frac{-t^i}{\log n}, \; \forall 1\leq i\leq N,
    \end{align*}
    so that $s^1 \geq \dots \geq s^N \geq 0$ and $|s| < 2(1-\varepsilon)$. Observe that for each fixed $n$, the function $f_n: \R^N \to [0,1]$ given by
    \begin{align*}
        f_n(v) \coloneqq \P_n\Big[\phi(x) - v \log n \notin I, \;\forall x \in \Lambda_n \cap \mathcal{B}^2(0, r_n n) \Big]
    \end{align*}
    is continuous; and hence, its supremum over the compact set $V \coloneqq \overline{B^N(0, 2(1-\varepsilon))}\cap \{v \in \R^N: v^1 \geq \dots \geq v^N \geq 0\}$ is attained at some point $s(n) \in V$. To complete the proof of the theorem, it therefore suffices to show that 
    \begin{align*}
        p_* \coloneqq \limsup_{n\to \infty} f_n(s(n)) = 0.
    \end{align*}
    By compactness of $V$, every subsequence $(s(n_k))_k$ has a further convergent subsequence $(s(n_{k_j}))_j$. It is enough to verify that $f_{n_{k_j}}(s(n_{k_j})) \to 0$ as $j \to \infty$, since every convergent subsequence of a sequence converging to $p_*$ has the same limit, implying $p_* = 0$.
    Taking this into account, and to simplify notation, we may assume directly that $s(n) \to s_*$ for some $s_* \in \overline{\B^N(0, 2(1-\varepsilon))}$ with $s^1_*\geq \dots \geq s_*^N \geq 0$.
   
    We now aim to apply Proposition \ref{prop:no_hole_harmonic} with $s = s_*$, so let $\alpha, \kappa$, and $\delta$ be as specified there. Note that since $\kappa < 2(1-\alpha)$, $n^\kappa \ll |\Pi_\alpha \cap \B|$ by our assumptions on the asymptotic behavior of $r_n$. By the domain Markov property and \eqref{eq:Greensfct_bounds}, since $I$ is bounded and with non-empty interior (in particular, contains an $N$-cube of positive side-length), it follows that for all $n$ sufficiently large
    \begin{align}
        f_n(s(n)) &\leq \P_{n}\big[\phi^B_{x_B}+ h_B  -s(n) \log n \notin I, \;\forall B \in \Pi_\alpha \cap \B\big]\nonumber\\
        &\leq \P_{n}\big[ N^{\alpha}_{s_*-\delta, s_*+\delta} \leq n^\kappa\big] 
        + \sup_{u \in [-\delta, \delta]^N} \P\Big[\sqrt{G_{\Lambda_{n^{\alpha}}}(0,0)}X + (1+o(1))u\log n \notin I\Big]^{n^\kappa} \nonumber\\
        &\leq \P_{n}\big[ N^{\alpha}_{s_*-\delta, s_*+\delta} \leq n^\kappa\big] 
        + \Big( 1- n^{-\frac{\delta^2 N}{2\alpha}(1+o(1))}\Big)^{n^\kappa}\nonumber\\
        &\leq \P_{n}\big[ N^{\alpha}_{s_*-\delta, s_*+\delta} \leq n^\kappa\big] 
        + e^{-n^{\kappa - \frac{\delta^2 N}{2\alpha}(1+o(1))}}, \label{eq:aux_to_upper_bound}
    \end{align}
    where $X$ denotes a standard normal vector in $\R^N$. The first summand in \eqref{eq:aux_to_upper_bound} vanishes in the limit $n\to \infty$ by Proposition \ref{prop:no_hole_harmonic}, while the second since $\delta < \sqrt{2 \alpha \kappa/N}$.    
\end{proof}

We now turn to the proof of Proposition \ref{prop:no_hole_harmonic}. The idea is to reduce the problem to studying the number of boxes for which $h_B/\log n$ is greater than a given vector, rather than working directly with $N^\alpha_{s-\delta, s+\delta}$. One must furthermore distinguish between the cases $s=0 \in \R^N$ and $s \neq 0$. The key input in the former case is the following lemma (proved below in the discussion of the $s=0$ case), which states that the number of \emph{positive} boxes,
\begin{align*}
    N^\alpha_+ \coloneqq \#\big\{B \in \Pi_\alpha \cap \B: h_B > 0\big\},
\end{align*}
cannot be much smaller than the total number of boxes in $\Pi_\alpha \cap \B$.
\begin{lemma}[Number of positive boxes] \label{lemma:pos_centers_estimate}
    Let $\alpha \in (0,1)$ and $\xi > 0$. Then
    \begin{align*}
        \lim_{n \to\infty} \P_n\big[ N^\alpha_+ \leq n^{2(1-\alpha) - \xi}\big] = 0.
    \end{align*}
\end{lemma}

In the case $s \neq 0$ (with $s^1 \geq \dots \geq s^N \geq 0$), the argument proceeds in two steps. First, we reduce to the case of an $L$-dimensional vector $(s^i)_{i=1}^L$ with all coordinates strictly positive, where $1 \leq L \leq N$ is chosen so that $s^L > 0$ and $s^{L+1} = 0$. In the second step, we use an external result from the proof of \cite[Theorem 1.3]{extremesGFF}, adapted to the vectorial setting\footnote{\cite[Theorem 1.3]{extremesGFF} denotes by $\Phi$ a scalar GFF and considers $\eta \in (0,1)$. Instead, let $\Phi, \eta$ both be vector valued with $\eta \in \R_+^L: 0< |\eta|<1$, and interpret relations such as $\Phi_B \geq c \eta$ for some $c\in \R_+$ coordinate-wise; $\eta^2$ as $|\eta|^2$. This modification does not affect (since coordinate processes are independent) the proof of the lower bound of the theorem in question up to trivial adjustments, and so, we will assume their results (properly adjusted) to be true for the vectorial case. \label{footnote:vect_thick_pts}} and our normalization of the field, which we now state. To distinguish from our set of parameters, we utilize upright Greek letters to refer to their parameters.
\begin{proposition}\label{prop:Dav}
    Let $\upalpha_1 \coloneqq \upalpha \in (1/2, 1)$, $\upkappa \in (0, 2(1-\upalpha))$, $\upeta \in \R_+^L$ with $0<|\upeta|<1$, $\K \geq 2$ such that $\K \upkappa \geq 4\upalpha$. Then, by setting $\Gamma_{\upalpha_1} \coloneqq \{B \in \Pi_{\upalpha_1}: B \subset \B\}$, we have that 
    \begin{align}
        \nonumber\P_{n}\bigg[\#\bigg\{B \in \Pi_{\upalpha/\K} \cap \B: h^{\primorial_{\upalpha/\K}}({x_B}) &\geq 2 \upeta  \upalpha \Big(1-\frac{1}{\K}\Big)^2\log n\bigg\} \leq n^{\upkappa + 2\upalpha (1-\frac{1}{\K})(1-|\upeta|^2)} \bigg]\\
        &\leq e^{-c\log^2 n} + \P_{n}\Big[\#\big\{B \in \Gamma_{\upalpha_1}: h^{\primorial_{\upalpha_1}}({x_B}) \geq 0\big\} \leq n^{\upkappa}\Big]\label{eq:Dav}
    \end{align} 
    for an appropriate constant $c = c(\upalpha, \upeta, \K) > 0$ as long as $\K$ is sufficiently large.
\end{proposition}
\begin{proof}
    This result follows directly from equations (2.10), (2.16), and the inequality right below it in \cite{extremesGFF}. Note that the fact that $\K \upkappa \geq 4\upalpha$ is needed for the power in (2.16) to always be positive, and that the change of the definition $\Gamma_{\upalpha_1}$ instead of $\Gamma_{\upalpha_1}=\Pi_{\upalpha_1}$ only affects (2.7) and the concluding inequality where (2.7) is used—we do not use either of those observations—but not the rest of the proof.
\end{proof}
\begin{remark}
    Note that the proof of \cite[Theorem 1.3]{extremesGFF} also controls the second summand in \eqref{eq:Dav}, but only for a specific choice of $\upkappa > 0$ (sufficiently small). This choice, however, is not sufficient for proving Proposition \ref{prop:no_hole_harmonic}, which explains the need for Lemma \ref{lemma:pos_centers_estimate}. Furthermore, the statement of Proposition \ref{prop:Dav} extends easily to the case $\upeta \in [0,\infty)^N$ with $0<|\upeta|<1$. A direct application to vectors $s$ with some zero entries, however, does not allow all parameter constraints (on $\alpha, \delta, \kappa$ matched with $\upeta, \upalpha, \upkappa, \K$) to be satisfied simultaneously (cf. Lemma \ref{lemma:inequalities_non_empty}).
\end{remark}

\textbf{The case $s=0$.} 
Let us first prove Proposition \ref{prop:no_hole_harmonic} for $s=0$ under the assumption of Lemma \ref{lemma:pos_centers_estimate}.
\begin{proof}[Proof of Proposition \ref{prop:no_hole_harmonic} for $s=0$ assuming Lemma \ref{lemma:pos_centers_estimate}] 
    Observe that 
    \begin{align*}
        N^{\alpha}_{s-\delta, s+\delta} = N^{\alpha}_{-\delta, \delta}
        &\geq \#\bigg\{B \in \Pi_{\alpha} \cap \B: 0 < \frac{h_B}{\log n} < \delta\bigg\}\geq  N_+^\alpha - \sum_{j=1}^N  N^{\alpha,j}_{> \delta},
    \end{align*} 
    where
    \begin{align*}
        N^{\alpha,j}_{>\delta} := \# \big\{ B \in \Pi_\alpha \cap \B: (h_B)^{j} > \delta \log n \big\}
    \end{align*}
    depends solely on the $j$-th coordinate of $h_B$. Using that the coordinate processes are identically distributed, we obtain that
    \begin{align*}
        \P_{n}\big[ N^{\alpha}_{-\delta, \delta} \leq n^\kappa\big] \leq \P_{n}\big[ N^{\alpha}_+ \leq (N+1) n^\kappa\big] + N \P_{n}\big[  N^{\alpha,1}_{>\delta} \geq n^\kappa\big].
    \end{align*}
    Write $\kappa = 2(1-\alpha) - \xi$ for $0< \xi < 2(1-\alpha)$. By Lemma \ref{lemma:pos_centers_estimate}, the first term in the latter expression is vanishing in the limit for any fixed choice of $\xi>0$. Therefore, to conclude this case, it only remains to show that for $\xi$ sufficiently small, the second term also tends to zero as $n\to \infty$. To this end, observe that Markov's inequality, together with \eqref{eq:Greensfct_bounds} (recall that $\Var_n[(h_B)^1] = G_{\Lambda_n}(x_B,x_B) - G_B(x_B, x_B)$) and the canonical tail-bound for a standard normal variable gives us
    \begin{equation}
    \begin{aligned}\label{eq:upper_bound_+high_pts}
        \P_{n}\big[  N^{\alpha,1}_{>\delta} \geq n^\kappa\big] &\leq n^{-\kappa} \sum_{B \in \Pi_\alpha \cap \B} \P_n \big[(h_B)^1 > \delta \log n \big] \leq n^{\xi -\frac{\delta^2}{2(1-\alpha)}(1-o(1))}.
    \end{aligned}
    \end{equation}
    The latter quantity is arbitrarily small provided $0<\xi < \delta^2$ for $n$ large. Thus, we need to choose parameters so that
    \begin{align*}
        2(1-\alpha) \xi < \delta^2 < \frac{2\alpha (2(1-\alpha) - \xi)}{N}
    \end{align*}
    holds. Take, for example, $\alpha = 1/2$, $\delta = 1/({2\sqrt{N}})$, and $\xi = 1/(8N)$.
\end{proof}

We conclude the discussion of the case $s=0$ by proving the remaining result Lemma \ref{lemma:pos_centers_estimate}
\begin{proof}[Proof of Lemma \ref{lemma:pos_centers_estimate}]
    Note that 
    \begin{align*}
        \E_n[N^\alpha_+] = \sum_{B \in \Pi_\alpha \cap\B} \P_n[h_B > 0] = \frac{1}{2^N} |\Pi_\alpha \cap \B|.
    \end{align*}
    Then, by Markov's inequality, the asymptotic behavior of $r_n$, and the i.i.d. nature of the coordinate processes, we have for any $0<s_n < r_n$ (to be specified later)
    \begin{align*}
        \P_n\big[ N^\alpha_+ &\leq n^{2(1-\alpha) - \xi}\big] \leq \frac{\E_n[(N^\alpha_+)^2] - \E_n[N^\alpha_+]^2}{(\E_n[N^\alpha_+] - n^{2(1-\alpha) - \xi})^2} \\
        &= \big(1 + o\big(n^{-\xi/2}\big)\big) \bigg(\frac{2^{2N}}{|\Pi_\alpha \cap\B|^2} \sum_{B, B' \in \Pi_\alpha \cap\B} \P_n \big[(h_B)^1, (h_{B'})^1 > 0 \big]^N - 1\bigg)\\
        &\leq \big(1 + o(1)\big) \Bigg(\frac{2^Ns_n^2}{r_n^2} + \frac{1}{|\Pi_\alpha \cap\B|^2} \sum_{\substack{B, B' \in \Pi_\alpha \cap\B\\ |x_B - x_{B'}| > s_n n}} \Big(4^N\P_n \big[(h_B)^1, (h_{B'})^1 > 0 \big]^N - 1\Big)\Bigg).
    \end{align*}
    Next, recall that for a centered normal vector $(X,Y)$ with $\E[X^2] = \E[Y^2] = 1$ and $\E[XY] = \rho$, we have the following (cf. \cite[III.9 Problem 14]{feller1968} and \cite{arcsin_formula}):
    \begin{align}\label{eq:arcsin_formula}
        \P[X>0, Y>0] = \frac{1}{4} + \frac{1}{2\pi} \arcsin \rho.
    \end{align}
    Furthermore, as a straightforward consequence of \eqref{eq:cov_harmext} and \eqref{eq:Greensfct_bounds}, we obtain for all $B\neq B' \in \Pi_\alpha$:
    \begin{equation}
    \begin{aligned} \label{eq:cov_bds_harmonic_ext}
        \Var_n[(h_B)^1] &\in \log (n^{1-\alpha}) + (-c,c), \\
        \E_n[(h_{B})^1(h_{B'})^1] &\in \Big(\log (n^{1-\alpha}) - \log\frac{|x_B - x_{B'}|}{n^\alpha}\Big) + (-c, c),
    \end{aligned}
    \end{equation}
    for an appropriate constant $c>0$ (depending only on $\upiota$). Hence, for all $B,B' \in \Pi_\alpha$ with $|x_B-x_{B'}| \geq s_n n$:
    \begin{align*}
        \rho_{B,B'} \coloneqq \frac{\E_n[(h_{B})^1(h_{B'})^1]}{\sqrt{\Var_n[(h_B)^1]\Var_n[(h_{B'})^1]}} \leq \frac{\log (1/s_n) + c}{(1-\alpha)\log n - c}. 
    \end{align*}
    By choosing $0< s_n = r_n/\sqrt{\log n}$—so that $\frac{\log (1/s_n)}{\log n} = o(1)$—and applying \eqref{eq:arcsin_formula} to $X = (h_B)^1/\sqrt{\Var_n[(h_B)^1]}$ and $Y= (h_{B'})^1/\sqrt{\Var_n[(h_{B'})^1]}$, we obtain that 
    \begin{align*}
        \frac{1}{|\Pi_\alpha \cap\B|^2} \sum_{\substack{B, B' \in \Pi_\alpha \cap\B\\ |x_B - x_{B'}| > s_n n}} \Big(4^N\P_n \big[(h_B)^1, (h_{B'})^1 > 0 \big]^N - 1\Big)
        \leq \frac{16N}{\pi} \frac{\log (1/s_n)}{(1-\alpha)\log n}
    \end{align*}
    Here, we also used the inequalities $(1+x)^N - 1 \leq 2N|x|$ and $|\arcsin x| \leq 2|x|$ valid for $x=o(1)$.

    The proof is complete, as the above combined with the assumption on $r_n$ implies that
    \begin{align*}
        \P_n\big[ N^\alpha_+ \leq n^{2(1-\alpha) - \xi}\big] = \mathcal{O} \bigg(\frac{\log \log n + \log(1/r_n)}{\log n}\bigg) \xrightarrow{n\to \infty} 0.
    \end{align*}
\end{proof}

\textbf{Case $s\neq 0$.} Recall that $s \in \B^N(0, 2(1-\varepsilon))$ for some $\varepsilon > 0$ is such that $s^1\geq \dots \geq s^N \geq 0$, and $1 \leq L \leq N$ denotes the largest index such that $s^L > 0$.

We begin by stating an auxiliary result providing parameters that meet the necessary constraints for the proof of Proposition \ref{prop:no_hole_harmonic} in this case.
\begin{lemma}\label{lemma:inequalities_non_empty}
    For $s$ as above (in particular, such that $s^L > 0$), there are positive constants $\upalpha, \K, \delta, \xi, \zeta >0$ such that the following inequalities are verified simultaneously:
    \begin{align}\tag{C1} \label{eq:C1}
        &\delta^2 < \frac{(s^L)^2}{4} \wedge \frac{2 \upalpha}{N \K} \Big( 2\Big(1-\frac{\upalpha}{\K}\Big) - \xi\Big);\\
        &\xi < \frac{\sum_{i=1}^{L-1} (s^i - \delta)^2 + (s^L + \delta)^2}{2(1-\upalpha/\K)} \wedge 2\Big(1-\frac{\upalpha}{\K}\Big); \tag{C2} \label{eq:C2} \\
        &\xi > \zeta + \sum_{i=1}^L \frac{(s^i - \delta)^2}{2\upalpha (1 - 1/\K)^3}; \tag{C3} \label{eq:C3}\\  
        &\K \geq \frac{2\upalpha}{1-\upalpha -\zeta/2}; \tag{C4} \label{eq:C4}\\
        &\sum_{i=1}^L (s^i - \delta)^2 < 4\upalpha^2 \Big(1- \frac{1}{\K}\Big)^4. \tag{C5}\label{eq:C5}
    \end{align}
\end{lemma}
Assuming this lemma, we proceed to prove Proposition \ref{prop:no_hole_harmonic}, postponing its proof to the end of this section.

\begin{proof}[Proof of Proposition \ref{prop:no_hole_harmonic} for $s\neq 0$ assuming Lemma \ref{lemma:inequalities_non_empty}] 
    Let $s$ (and $L$) be as specified above, and let $\upalpha, \K, \delta, \xi, \zeta > 0$ be as in Lemma \ref{lemma:inequalities_non_empty}. Define $\alpha \coloneqq \upalpha/\K \in (0,1)$ and $\kappa\coloneqq 2(1-\alpha) - \xi$. In particular, by \eqref{eq:C2} we have $\kappa \in (0,2(1-\alpha))$, and by \eqref{eq:C1} we have $0 < \delta^2 < 2\alpha \kappa/N$, as desired. We start by reducing to the case when all the coordinates in question are away from $0$. To this end, define  
    \begin{align*}
        \Sc_L(\delta) \coloneqq \bigg\{B \in \Pi_{\alpha} \cap \B: s^i - \delta < \frac{(h_B)^i}{\log n} < s^i+\delta, \, \forall i \leq L \bigg\},
    \end{align*}
    and observe that
    \begin{align*}
        N^{\alpha}_{s-\delta, s+\delta}
        &= \#\bigg\{B \in\Sc_L(\delta): \frac{(h_B)^i}{\log n} \in (- \delta, \delta), \; \forall L+1 \leq i \leq N\bigg\}\\
        &\geq \# \Sc_L(\delta) - \sum_{j=L+1}^N \#\bigg\{B \in \Sc_L(\delta): \frac{|(h_B)^j|}{\log n} \geq \delta\bigg\}.
    \end{align*}
    Fully analogously to \eqref{eq:upper_bound_+high_pts} (by Markov's inequality (conditionally on $S_L(\delta)$) and \eqref{eq:Greensfct_bounds}), we obtain 
    \begin{align*}
        \P_{n}\big[ N^{\alpha}_{s-\delta, s+\delta} \leq n^\kappa\big]  &- \P_{n}\big[ \# \Sc_L(\delta) \leq 2n^\kappa\big] \\
        &\leq \sum_{j=L+1}^N \P_{n}\bigg[\#\bigg\{B \in \Sc_L(\delta): \frac{|(h_B)^j|}{\log n} \geq \delta\bigg\} \geq \frac{\# \Sc_L(\delta)}{2N} \bigg]\\
        &\leq 4N(N-L) n^{-\frac{\delta^2}{2(1-\alpha)}(1-o(1))}.
    \end{align*}
    The latter expression is arbitrarily small, hence, it only remains to bound $\P_{n}[ \# \Sc_L(\delta) \leq 2n^\kappa]$ appropriately. 

    For notational convenience, in the remainder of the proof we restrict to the $L$-vectorial field and write $s$ for $(s^i)_{i=1}^L$. For $v \in \R^L$, we define
    \begin{align*}
        N^\alpha_{> v} \coloneqq \#\bigg\{ B \in \Pi_\alpha \cap \B: \frac{h_B}{\log n} > v\bigg\}.
    \end{align*}
    Recall that $\kappa = 2(1-\alpha) - \xi$, and observe that
    \begin{align*}
        \# \Sc_L(\delta) \geq N^\alpha_{> s - \delta} - \sum_{j=1}^L N^\alpha_{>s - \delta + 2\delta \ind_{\{j\}}}.
    \end{align*}
    Therefore, by reiterating the above procedure, i.e., applying Markov's inequality and \eqref{eq:Greensfct_bounds}, we get
    \begin{align*}
        \P_{n}\big[ \# \Sc_L(\delta) \leq 2n^\kappa\big] - \P_{n}\big[ N^\alpha_{> s - \delta} \leq (L+2) n^\kappa\big] 
        &\leq L \max_{j=1}^L \P_{n}\big[N^\alpha_{>s - \delta + 2\delta \ind_{\{j\}}} > n^\kappa\big] \\
        &\leq \max_{j=1}^L n^{\xi - \frac{1}{2(1-\alpha)} \sum_{i=1}^L(s^i - \delta + 2\delta \ind_{\{i=j\}})^2 + o(1)}\\
        &\leq n^{\xi - \frac{1}{2(1-\alpha)} \sum_{i=1}^L(s^i - \delta + 2\delta \ind_{\{i=L\}})^2 + o(1)}.
    \end{align*}
    The latter vanishes in the limit as by \eqref{eq:C2},
    \begin{align*} 
        0 <\xi < \frac{\sum_{i=1}^{L-1} (s^i - \delta)^2 + (s^L + \delta)^2}{2(1-\alpha)} \wedge 2(1-\alpha).
    \end{align*}

    To bound $\P_{n}\big[ N^\alpha_{> s - \delta} \leq (L+2) n^\kappa\big]$, we want to apply Proposition \ref{prop:Dav}. To this end, let
    \begin{align*}
      \upkappa \coloneqq 2(1-\upalpha) - \zeta &&\text{ and } && \upeta \coloneqq \frac{s-\delta}{2\upalpha (1-1/\K)^2} \in \R_+^L.
    \end{align*}
    Here we used that $\delta < s^L/2$ by \eqref{eq:C1}. Note that by Lemma \ref{lemma:pos_centers_estimate}, the last term in \eqref{eq:Dav} of Proposition \ref{prop:Dav} converges to zero as $n\to \infty$. Therefore, we indeed obtain the desired convergence
    \begin{align*}
        \P_{n}\big[ N^\alpha_{> s - \delta} \leq (L+2) n^\kappa\big] \xrightarrow{n\to\infty} 0
    \end{align*}
    provided that 
    \begin{align*} 
        \kappa < \upkappa + 2\upalpha \Big(1-\frac{1}{\K}\Big)(1-|\upeta|^2),  &&\K \upkappa \geq 4\upalpha, &&\sum_{i=1}^L (s^i - \delta)^2 < 4\upalpha^2 \Big(1- \frac{1}{\K}\Big)^4.
    \end{align*}
    These inequalities are in turn yielded by \eqref{eq:C3}, \eqref{eq:C4}, and \eqref{eq:C5}, respectively.
\end{proof}

To complete the proof of Theorem \ref{thm:2Dno_hole}, it only remains to establish Lemma \ref{lemma:inequalities_non_empty}.
\begin{proof}[Proof of Lemma \ref{lemma:inequalities_non_empty}]
 To ensure \eqref{eq:C3}, we choose 
    \begin{align*}
        \xi \coloneqq 2\zeta + \sum_{i=1}^L \frac{(s^i-\delta)^2}{2\upalpha (1-1/\K)^3}.
    \end{align*}
    By plugging this expression into \eqref{eq:C1} and setting $q \coloneqq (\K(1-1/\K)^3)^{-1}$, we see that the assumption on $\delta \in (0, s^L/2)$ is equivalent to
    \begin{align*}
        \delta^2 \Big(1+ \frac{qL}{N}\Big) - \delta \sum_{i=1}^L \frac{2q s^i}{N} - \frac{1}{N\K} \bigg(4\upalpha \Big(1-\frac{\upalpha}{\K} - \zeta\Big) - q \K \sum_{i=1}^L (s^i)^2 \bigg) < 0.
    \end{align*}
    Thus, if we additionally can choose $\upalpha, \K$ such that 
    \begin{align} \tag{C5'}
        \label{eq:C5'}
        \sum_{i=1}^L (s^i)^2 = |s|^2 \leq 4\upalpha \Big(1-\frac{\upalpha}{\K} - \zeta\Big) \Big/ \bigg(\frac{N}{\K^r} + \frac{1}{(1-{1}/{\K})^{3}}\bigg)
    \end{align}
    for some $r \in (0,1/2)$, then, as long as $\K>0$ is sufficiently large (depending on $N, s^L$),
    \begin{align*}
        \delta \coloneqq \frac{|s|(1- \K^{-10})}{1+ qL/N} \Bigg( \frac{\sum_{i=1}^L qs^i}{N|s|} + \sqrt{\bigg(\frac{\sum_{i=1}^L qs^i}{N|s|}\bigg)^2 + \frac{1+qL/N}{N \K} \bigg(\frac{4\upalpha (1-\upalpha/\K - \zeta)}{|s|^2} - q \K \bigg)} \Bigg) 
    \end{align*}
    indeed fulfills \eqref{eq:C1}, and, moreover, (since $|s|<2$)
    \begin{align*}
        |s| \frac{1 - L/(N\K)}{\K^{(1+r)/2}} \leq \delta \leq |s|\bigg(q + \sqrt{q^2 + \frac{4(1+q)}{\K N|s|^2}}\bigg) \leq \frac{2|s|}{\K} + \frac{4}{\sqrt{N \K}} \leq \frac{8}{\sqrt{N\K}} < \frac{s^L}{2}.
    \end{align*} 
    Let 
    \begin{align}\label{eq:upalpha_choice}
        \upalpha \coloneqq 1 - \frac{1}{\K^{(1+r)/2}} \frac{s^L}{2},
    \end{align}
    then \eqref{eq:C4} is satisfied for all $\K$ large enough (depending on $s^L$) and $\zeta \leq \K^{-(1+r)/2} s^L/2$. Note that with this choice of $\alpha$—since $r < 1/2 < (1+r)/2$, $\delta \leq s^L/2$, and hence
    \begin{align}\label{eq:s-delta}
        \sum_{i=1}^L (s^i - \delta)^2 \leq |s|^2,
    \end{align}
    for all $\K$ sufficiently large (depending on $s^L$), \eqref{eq:C5'} implies \eqref{eq:C5}, and is in turn yielded by the inequality $|s| \leq 2(1- N \K^{-r})$. Recall, however, that $|s| \leq 2(1-\varepsilon)$, so by making $K$ even larger if needed depending on $\varepsilon$, we indeed obtain $|s| \leq 2(1- N \K^{-r})$. Hence, to complete the proof, it suffices to show that our choice of parameters (in particular, we may work under \eqref{eq:C5'}) fulfills 
    \begin{align*}
        \sum_{i=1}^L \frac{(s^i-\delta)^2}{2\upalpha (1-1/\K)^3} 
        < \frac{\sum_{i=1}^{L-1} (s^i - \delta)^2 + (s^L + \delta)^2}{2(1-\upalpha/\K)} \wedge 2\Big(1-\frac{\upalpha}{\K}\Big)
    \end{align*}
    (since then \eqref{eq:C2} follows if we set $\zeta$ to be the minimum between $\upalpha/\K$ and the $1/4$ of the difference of the expressions on the right- and left-hand sides, respectively). First observe that for $\K$ sufficiently large, similarly to the above, \eqref{eq:C5'} implies that $|s|^2 \leq 4\alpha (1-1/\K)^4$, and thus by \eqref{eq:s-delta}, it is left to establish that
    \begin{align*}
        \sum_{i=1}^L \frac{(s^i-\delta)^2}{2\upalpha (1-1/\K)^3} 
        < \frac{\sum_{i=1}^{L-1} (s^i - \delta)^2 + (s^L + \delta)^2}{2(1-\upalpha/\K)}.
    \end{align*}
    The latter inequality indeed holds true: by first applying \eqref{eq:upalpha_choice} and \eqref{eq:s-delta}, then \eqref{eq:C5'}, and finally the above lower bound for $\delta$, we obtain that
    \begin{align*}
        \sum_{i=1}^L (s^i-\delta)^2 \Big(\frac{1-\upalpha/\K}{\upalpha (1-1/\K)^3} - 1\Big) < \frac{s^L |s|^2}{\K^{(1+r)/2}} 
        \leq \frac{2 s^L |s| (1 - (N/4) \K^{-r})}{\K^{(1+r)/2}}
        < 4 s^L \delta  
    \end{align*}
    for all $\K$ sufficiently large (depending on $s^L, N$) as desired.
\end{proof}


\section[Entropic repulsion of the norm]{Entropic repulsion of the norm of the vector-valued GFF conditioned to avoid a ball}
\label{sec:repulsion}

The goal of this section is to prove the entropic repulsion phenomenon (cf. Theorem \ref{thm:repulsion}) of an $N$-vectorial Dirichlet Gaussian free field $\phi$ conditioned to avoid a ball. As a warm-up, we discuss the case of a half-line for $N=1$ and its direct consequences for the case of a bounded ball (interval if $N=1$).

\subsection{Warm-up: Avoiding a half-line and its consequences for the ball case}

Let $b\in \R$ and $\phi$ be a scalar Dirichlet GFF on $\Lambda_n$. We now discuss its law conditioned on $\Omega^{b,+}_{D_n} \coloneqq \{\phi(x) \geq b, \; \forall x \in D_n\}$. 

\begin{proposition}
\label{prop:halfline_norm}
    Let $\phi$ be a scalar Dirichlet GFF on $\Lambda_n$. Then, for any $\beta > 0$, 
    \begin{align}
        \lim_{n \rightarrow \infty} &\frac{\log \P_n[ \Omega_{D_n}^{b,+}]}{\log^2 n} = -\frac{2}{\pi} \:\cpc_\Lambda(D); \label{eq:1sD_prob.cond.event} \\
        \lim_{n \rightarrow \infty}\sup_{x \in D_n} &\P_n \big[|\phi(x)- 2 \log n \big| \geq \beta\log n ~|~ \Omega^{b,+}_{D_n}\big] = 0. \label{eq:1sD_norm}
    \end{align}
    Here, $\cpc_\Lambda(D)$ is the relative capacity of $D$ with respect to $\Lambda$ defined by $$\cpc_\Lambda(D) = \frac{1}{2} \inf\{\norm{\nabla f}^2_{L^2(\Lambda)}: f \in H^1_0 (\Lambda), \,f \geq 1 \text{ on } D\}.$$ 
\end{proposition}
\begin{proof}
     For $b=0$, the two statements are Theorem 3 and Theorem 4 in \cite{entropRep2D}\footnote{Recall that our GFF is normalized such that $\E_n[\phi(0)^2] \sim \log n$ as $n\to \infty$, while \cite{entropRep2D} uses the random walk normalization of the Laplacian, which explains the difference in \eqref{eq:1sD_norm}: in their setting, $\E_n[\tilde\phi(0)^2] \sim g\log n$ with $g = \frac{2}{\pi}$. As for \eqref{eq:1sD_prob.cond.event}, there is a typo in the reference—under the random walk normalization, their relative capacity $\tilde{\cpc}_\Lambda(D)$ should be given by $\frac{1}{2d} \cpc_\Lambda(D)$, rather than simply $\cpc_\Lambda(D)$, which makes $4g\tilde{\cpc}_\Lambda(D) = 2/\pi$. This is consistent with our result, as the probability of the hard-wall event should not depend on the global normalization of the field. Also note that in other literature, the relative capacity is sometimes defined as $2\cpc_\Lambda(D)$. \label{footnote:typo_cond_event}}, respectively.
     Their proofs carry over to the case of arbitrary $b \in \R$ with only trivial adjustments.
\end{proof}

The above theorem has some rather direct consequences for an $N$-vectorial ($N \in \N$) Dirichlet GFF $\phi$ on $\Lambda_n$ conditioned to avoid a ball $I$: $I$ is a proper interval $(a,b) \subseteq \R$ when $N=1$, or a centered ball of radius $R>0$—$\B^N(0,R)\subset \R^N$—when $N\geq 2$. We also define the event
\begin{align} \label{eq:Omega_def}
    \Omega^{I}_{D_n} \coloneqq \{\phi(x) \notin I, \; \forall x \in D_n\}.
\end{align}
\begin{corollary}
\label{cor:halfline_implications}
    Let $\phi$ be an $N$-vectorial Dirichlet Gaussian free field on $\Lambda_n$. Then, for any $\beta > 0$,
    \begin{align}
        \liminf_{n \rightarrow \infty} &\frac{\log \P_n[ \Omega_{D_n}^{I}]}{\log^2 n} \geq -\frac{2}{\pi} \:\cpc_\Lambda(D);  \label{eq:lowerbound_condevent}\\
        \lim_{n \rightarrow \infty}\sup_{x \in D_n} &\P_n \big[|\phi(x)| \geq (2+\beta) \log n ~\big|~ \Omega^{I}_{D_n} \big] = 0; \label{eq:norm_upperbound}.
    \end{align}
\end{corollary}
\begin{proof}
    Note that $\Omega_{D_n}^I \supset \{\phi^1(x) \geq R\vee|b|, \; \forall x \in D_n\}$, which by \eqref{eq:1sD_prob.cond.event} readily implies the first result. 
    
    As for the remaining statement, we first establish it for $N=1$. By the FKG property of the conditioned field (Proposition \ref{prop:FKG_condGFF}) and by symmetry, 
    \begin{align*}
        \P_n\bigg[\frac{|\phi(x)|}{\log n} \geq 2+\beta ~\bigg|~ \Omega^{I}_{D_n}\bigg] \leq &\P_n\bigg[\frac{\phi(x)}{\log n} \geq 2+\beta ~\bigg|~ \Omega^{b,+}_{D_n}\bigg] 
        + \P_n\bigg[\frac{\phi(x)}{\log n} \geq 2+\beta ~\bigg|~ \Omega^{-a,+}_{D_n}\bigg].
    \end{align*}
    
    For $N\geq 2$, observe that the complement of the ball $\B^N(0, (2 + \beta) \log n) \subset \R^N$ can be covered by finitely many (the optimal such number $C(N, \beta)$ depends on $N$ and $\beta$, and is non-increasing in the latter) rotations of the set $\{y \in \R^N: |y^1| \geq (2 + \beta/2) \log n\}$. Thus, since the law of $\phi$ and $\Omega^I_{D_n}$ are invariant under a (deterministic) global orthogonal transformation, we get that
    \begin{align*}
        \P_{n,N}\bigg[\frac{|\phi(x)|}{\log n} \geq 2+\beta ~\bigg|~ \Omega^{I}_{D_n}\bigg] &\leq C(N, \beta) \, \P_{n,N}\bigg[\frac{|\phi^1(x)|}{\log n} \geq 2 + \frac{\beta}{2} ~\bigg|~ \Omega^{I}_{D_n}\bigg].
    \end{align*}
    
    The goal now is to estimate the right hand-side by an appropriate expression involving only scalar fields. To this end, we define a collection of forbidden sets for $\phi^1$: for $y \in D_n$, let $$r_y \coloneqq \sqrt{0 \vee \bigg(R^2 - \sum_{j=2}^N (\phi^j(y))^2\bigg)}$$ and set $A = (A_y)_y \coloneqq ((-r_y, r_y))_y$. We further note that given $(\phi^{j})_{j=2}^N$, $\Omega_{D_n}^I$ can be viewed as a $\sigma(\phi^1)$-measurable event $\Omega_{D_n}^{A}(\phi^1) \coloneqq \{\phi^1(y) \notin A_y, \;\forall y \in D_n\}$. By independence of the coordinate processes (under $\P_{n,N}$), we obtain that 
    \begin{align*}
        \P_{n,N}\bigg[\frac{|\phi^1(x)|}{\log n} \geq 2 + \frac{\beta}{2} ~\bigg|~ \Omega^{I}_{D_n}, (\phi^j)_{j=2}^N \bigg] = \tilde{\P}_{n} \bigg[\frac{|\tilde \phi(x)|}{\log n} \geq 2 + \frac{\beta}{2}  ~\bigg|~ \Omega^{A}_{D_n}(\tilde \phi)\bigg],
    \end{align*}
    where $\tilde \phi \sim \tilde{\P}_n$ is an independent scalar Dirichlet GFF on $\Lambda_n$. Note that for any possible realization of $A$, $\tilde\P_n[\cdot | \Omega^{A}_{D_n}]$ satisfies the FKG property by Proposition \ref{prop:FKG_condGFF}. Thus, first by using symmetry and then the FKG property, we obtain that for any possible realization of $A$
    \begin{align*}
        \tilde{\P}_{n} \bigg[\frac{|\tilde \phi(x)|}{\log n} \geq 2 + \frac{\beta}{2} ~\bigg|~ \Omega^{A}_{D_n}(\tilde \phi)\bigg] \leq 2\tilde{\P}_{n} \bigg[\frac{\tilde \phi(x)}{\log n} \geq 2 + \frac{\beta}{2} ~\bigg|~ \Omega^{R, +}_{D_n}\bigg].
    \end{align*}
    Altogether, uniformly in $x \in D_n$,
    \begin{align*}
        \P_{n,N}\bigg[\frac{|\phi(x)|}{\log n} \geq 2+\beta ~\bigg|~ \Omega^{I}_{D_n}\bigg] &\leq 2C(N, \beta) \tilde{\P}_{n} \bigg[\frac{\tilde \phi(x)}{\log n} \geq 2 + \frac{\beta}{2} ~\bigg|~ \Omega^{R, +}_{D_n}\bigg],
    \end{align*}
    which together with \eqref{eq:1sD_norm} readily implies \eqref{eq:norm_upperbound}. 
\end{proof}

\subsection{Avoiding a ball}
Let $n, N \in \N$. If $N = 1$, set $I = (a, b)$ for $a < b \in \R$; otherwise, let $I = \B^N(0, R) \subset \R^N$ for some $R > 0$. Recall that for an $N$-vectorial Dirichlet GFF $\phi$ on $\Lambda_n$, we set
\begin{align*}
    \Omega^{I}_{D_n} \coloneqq \{\phi(x) \notin I, \; \forall x \in D_n\}.
\end{align*}
The objective of this subsection is to show that the norm of $\phi$ upon conditioning on $\Omega^I_{D_n}$ is approximately $2 \log n$ for all but negligibly many sites in $D_n$.  
\begin{theorem}[Entropic repulsion of the norm of GFF by a ball]
\label{thm:repulsion_ball}
    Let $\phi$ be an $N$-vectorial Dirichlet GFF on $\Lambda_n$.
    For any $\beta > 0$ and $\varepsilon > 0$, the set of vertices $x \in D_n$ not satisfying
    \begin{align}
    \label{eq:norm_under_cond}
        \P_n\bigg[\frac{|\phi(x)|}{\log n} \in (2 -\beta, 2 + \beta) ~\bigg|~ \Omega_{D_n}^I\bigg] \geq 1 - \varepsilon
    \end{align}
    is of cardinality at most $o(n^2) \ll |D_n|$ as $n \rightarrow \infty$. 
\end{theorem}
Throughout this subsection, we extensively use the notation of $\alpha$-subgrids/boxes introduced in Section \ref{sec:setup} (see also Figure \ref{fig:setup}). 

The two key ingredients (proven in Subsection \ref{subsec:key_ingredients_repulsion} below) in addition to Corollary \ref{cor:halfline_implications}, to establish that the norm of the field under conditioning must be at least $(2 - \beta) \log n$ for ``almost" all vertices of $D_n$ are the following. 
Firstly, we prove that for a given $\beta > 0$, the number of boxes in $\Pi_\alpha(x_0)$ with $|h_B| < (2 - \beta) \log n$ is negligible on $\Omega^I_{D_n}$ if $\alpha$ is close enough to one. 
\begin{proposition}[Number of \emph{low} boxes is negligible]
\label{prop:few_low_boxes}
    Let $\beta > 0$. There exists $\alpha_0(\beta) \in (0,1)$ sufficiently close to one such that for all $\alpha_0(\beta) \leq \alpha < 1$, $u \in (0, 2(1-\alpha))$, and all $x_0 \in \Z^2$, 
    \begin{align*}
        \P_n\big[\Omega^I_{D_n} \cap \{N^\alpha(\beta, x_0) \geq n^u\} \big] \leq e^{-(\frac{2}{\pi} \cpc_\Lambda(D) + 1) \log^2 n},
    \end{align*}
    for all $n$ sufficiently large (depending on $\beta, u, \alpha$). Here, 
    \begin{align} \label{eq:Nalpha_low_h}
        N^\alpha(\beta, x_0) \coloneqq \#\big\{B \in \Pi_\alpha(x_0): |h_B| < (2 - \beta) \log n\big\}.
    \end{align}
\end{proposition}
\noindent Secondly, to be able to get back to $\phi$, we analyze the family $(\phi^B)_{B \in \Pi_\alpha(x_0)}$. To this end, for $\eta > 0$ and $T \subset \Lambda_{n^\alpha/2}$, we define 
\begin{align}\label{eq:Malpha_high_fl}
    M^\alpha_{T}(\eta, x_0) \coloneqq \# \Big\{B \in \Pi_\alpha(x_0): \sup_{y \in T}|\phi^B({x_B + y})| > \eta \log n \Big\},
\end{align}
which intuitively corresponds to the number of boxes with high fluctuations around centers. When $T = \{0\}$, we omit the subscript. We show that $M^\alpha_T(\eta, x_0)$ is small when $\eta$ and the cardinality of $T$ (depending on $\eta$) are sufficiently small.
\begin{proposition}[Number of boxes with high fluctuations around centers is negligible]
\label{prop:few_high_fluct_boxes}
    Let $\alpha \in (0,1)$. There exists $\eta_0(\alpha, N) > 0$ such that for all $x_0 \in \Z^2$, for all $0 <\eta \leq \eta_0(\alpha, N)$ and $T \subset \Lambda_{n^\alpha/2}$ of cardinality at most $n^{\frac{\eta^2}{8 \alpha N}}$, 
    \begin{align*}
        \P_n\Big[\Omega^I_{D_n} \cap \big\{M^\alpha_{T}(\eta, x_0) \geq n^{-\eta^3} |\Pi_\alpha(x_0)|\big\} \Big] \leq e^{-n^{c}}
    \end{align*}
    for any $0<c = c(\alpha, \eta)<2(1-\alpha) - \eta^3$.
\end{proposition}

The statement of Theorem \ref{thm:repulsion_ball} now follows from the fact that the random sets of \emph{good} centers 
$$\G^\alpha(x_0, \beta, \eta) \coloneqq \Big\{x_B: B \in \Pi_\alpha(x_0), |h_B| \geq (2 - \beta) \log n, \;|\phi^B({x_B})| \leq \eta \log n \Big\}$$ 
for a fixed $\alpha \in (0,1)$ (sufficiently close to one) and different $x_0 \in \Lambda_{n^\alpha}$ produce disjoint deterministic sets of vertices satisfying \eqref{eq:norm_under_cond}, whose union is of cardinality $|D_n|(1 - o(1))$. 
\begin{proof}[Proof of Theorem \ref{thm:repulsion_ball}]
    Note that \eqref{eq:norm_upperbound} allows us to reduce the claim of the theorem to showing that 
    \begin{align}
        \label{eq:simplified.norm_under_cond}
        \P_n \big[|\phi(x)| \geq (2 - \beta) \log n ~\big|~ \Omega^I_{D_n}\big] \geq 1 -\varepsilon
    \end{align}
    for all but $o(n^2)$ points $x \in D_n$.
    Let $\alpha \in (0,1)$, $\eta > 0$, and notice that 
    \begin{align*}
        |\G^\alpha(x_0, \beta/2, \eta)| \geq |\Pi_\alpha(x_0)| - N^\alpha(\beta/2, x_0) - M^\alpha(\eta, x_0).
    \end{align*}
    Hence, by Propositions \ref{prop:few_low_boxes} and \ref{prop:few_high_fluct_boxes} and \eqref{eq:lowerbound_condevent}, we can choose $\alpha$ sufficiently close to $1$ and $\eta >0$ sufficiently small—for example, $\alpha = \alpha_0(\beta/2)$, $\eta \coloneqq (\beta/2) \wedge \eta_0(\alpha, N) \wedge (1-\alpha)^{1/3}$—so that
    \begin{align*}
        \E_n\big[ &|\G^\alpha(x_0, \beta/2, \eta)| ~\big|~ \Omega^I_{D_n}\big] \\
         &\geq \left(|\Pi_\alpha(x_0)| - 2 n^{2(1-\alpha) - \eta^3}\right) \P_n \big[N^\alpha(\beta/2, x_0)\vee M^\alpha (\eta, x_0) \leq n^{2(1-\alpha) - \eta^3}~\big|~ \Omega^I_{D_n}\big] \\
        &\geq \left(|\Pi_\alpha(x_0)| - 2 n^{2(1-\alpha) - \eta^3} \right) \big(1 - 2e^{-\frac{1}{2}\log^2 n}\big)
        = |\Pi_\alpha(x_0)| \big(1 - \mathcal{O}(n^{- \eta^3})\big)
    \end{align*}
    for all $n$ sufficiently large. On the other hand,
    \begin{align*}
        \E_n\big[ |\G^\alpha(x_0, \beta/2, \eta)| ~\big|~ \Omega^I_{D_n}\big] \leq \sum_{B \in \Pi_\alpha(x_0)} \P_n\big[ |\phi({x_B})| \geq (2 - \beta) \log n ~\big|~ \Omega^I_{D_n}\big].
    \end{align*}
    The two observations imply that for all sufficiently large $n$ (depending only on $\beta$ and $\varepsilon$), the set 
    \begin{align*}
        \Cc_n(x_0) \coloneqq \Big\{x_B: B \in \Pi_\alpha(x_0), \, \P_n\big[ |\phi({x_B})| \geq (2 - \beta) \log n ~\big|~ \Omega^I_{D_n}\big] \geq 1 - \varepsilon\Big\}
    \end{align*} 
    contains at least $|\Pi_\alpha(x_0)|\big(1 - n^{-\frac{\eta^3}{2}}\big)$ sites. 
    Since by construction $(\Cc_n(x_0))_{x_0 \in \Lambda_{n^\alpha} \setminus \partial_{\text{in}} \Lambda_{n^\alpha}}$ are pairwise disjoint and 
    \begin{align*}
    \tilde D_n \coloneqq \big\{x \in D_n: \P_n\big[|\phi(x)| \geq (2 - \beta) \log n ~\big|~ \Omega^I_{D_n}\big] \geq 1 - \varepsilon\big\} \supset \Cc_n(x_0)
    \end{align*} 
    for any $x_0$, as desired, we obtain that 
    \begin{align*}
        |\tilde D_n| \geq \sum_{x_0 \in \Lambda_{n^\alpha} \setminus \partial_{\text{in}} \Lambda_{n^\alpha}} |\Pi_\alpha(x_0)|\big(1 - n^{- \frac{\eta^3}{2}}\big) = |D_n| (1 - o(1)).
    \end{align*}
\end{proof}

\subsection{Proof of the key ingredients}
\label{subsec:key_ingredients_repulsion}

We now prove the two aforementioned key ingredient, Propositions \ref{prop:few_low_boxes} and \ref{prop:few_high_fluct_boxes}.

\subsubsection{Number of low boxes is negligible} 
To show that there are not many low boxes—$B\in \Pi_\alpha(x_0)$ with $|h_B| < (2 - \beta) \log n$—when $\alpha\in(0,1)$ is sufficiently close to one, we exploit two facts: that the harmonic extension $h^{\primorial_\alpha(x_0)}$ is almost constant close to the centers of boxes $B$, and that the field $\phi^B$ does not have a hole around a given point ($-h_B(\omega)$) in its range if the latter is bounded away from the boundary of the range (which is the case for low boxes)—cf. Theorem \ref{thm:2Dno_hole}. We then conclude using the domain Markov property of $\phi$. For better readability, we suppress the dependence on $x_0$. 
\begin{proof}[Proof of Proposition \ref{prop:few_low_boxes}]
    Let $l \coloneqq (b-a) \ind_{\{N=1\}} + 2R \ind_{\{N\geq 2\}}$ and define
    \begin{align*}
        F \coloneqq \bigcup_{B \in \Pi_\alpha} \bigg\{\sup\Big\{|h^{\primorial_\alpha}(x) - h_B|: x \in D_n, \, |x - x_B| < n^\alpha\log^{-2} n \Big\} \geq \frac{l}{4}\bigg\}.
    \end{align*}
    Note that all $x_B, x$ in the above expression are at least at distance $\upiota n/4$ from the outer boundary of $\Lambda_n$. Therefore, by \eqref{eq:harmdiff_bound}, we obtain that for all $B \in \Pi_\alpha$ and $x \in D_n$ such that $|x-x_B| < n^\alpha \log^{-2}n$, the variance of the $i$th coordinate—$\Var_n[(h^{\primorial_\alpha}(x) - h_B)^i]$—is bounded by $c\log^{-4} n$ for an absolute constant $c>0$. Combined with the union bound, this implies that uniformly in $x_0 \in \Z^2$,
    \begin{align*}
        \P_n[F] &\leq n^2 N \sup \bigg\{\P_n \bigg[ \big|(h^{\primorial_\alpha}(x) - h_B)^1 \big| \geq \frac{l}{4\sqrt{N}} \bigg]: B \in \Pi_\alpha, \,
        x \in D_n, \, |x-x_B| < \frac{n^\alpha}{\log^{2} n} \bigg\} \\&\leq 
        n^2 N e^{-\frac{l^2}{32 c N}\log^4 n} = o\big(e^{-(\frac{2}{\pi} \cpc_\Lambda(D) + 1) \log^2 n}\big)
    \end{align*}
    for all $n$ sufficiently large. We thus can work on the complement of $F$. 
    
    We define $I'$ as $(a + l/4, b - l/4)$ if $N = 1$, and as $\B^N(0, R/2)$ if $N \geq 2$, to compensate for the deviations of $h^{\primorial_\alpha}$ allowed on $F^c$; and denote the set of low boxes—boxes contributing to $N^\alpha(\beta)$—by $\mathcal{P}^\alpha(\beta)$. By looking only at the vertices $x$ sufficiently close to the centers of low boxes, by the domain Markov property of $\phi$ (Lemma \ref{lemma:domainMP}), for any $\omega \in F^c \cap \{N^\alpha(\beta) \geq n^u\}$, we obtain
    \begin{align*}
        \P_n\big[\Omega^I_{D_n} ~\big|~\F_{\primorial_\alpha}\big](\omega) &\leq \max_{B \in \mathcal{P}^\alpha(\beta)(\omega)} \P_n \Big[\phi^B(x) + h_B(\omega) \notin I', \;\forall x \in \mathcal{B}^2\big(x_B, n^\alpha\log^{-2}n\big) \Big]^{n^u}\\
        &\leq \sup_{t\in \B^N (0, (2-\beta) \log n)} \tilde{\P}_{n^\alpha}\Big[\tilde \phi(x) + t \notin I', \;\forall x \in \mathcal{B}^2 \big(0, n^\alpha\log^{-2}n\big)\Big]^{n^u} \eqqcolon p^{n^{u}}.
    \end{align*}
    Here, $\tilde \phi \sim \tilde{\P}_{n^\alpha}$ is an $N$-vectorial Dirichlet GFF on $\Lambda_{n^\alpha}$. By choosing 
    \begin{align} \label{eq:alpha_choice}
        \alpha_0(\beta) \coloneqq \frac{2 - \beta}{2 - \beta/2} \in (0,1),
    \end{align} 
    the proof is complete by Theorem \ref{thm:2Dno_hole}. Indeed, for any $\alpha \in (\alpha_0(\beta), 1)$, the latter implies that $p$ is, for instance, smaller than $1/2$ for all $n$ sufficiently large; which taken to the power $n^u$ is of even lower order than we wanted to show.
\end{proof}

\subsubsection{Number of boxes with high fluctuations around centers is negligible}
To prove that the number of boxes with high fluctuations around the centers $x_B \in D_n$ is negligible, we observe (using the domain Markov property of GFF) that it can be rewritten as a sum of i.i.d. indicator functions of events of low probability. We then use Bernstein's inequality to estimate its tail probabilities.   
For better readability, we suppress the dependence on $x_0$. 
\begin{proof}[Proof of Proposition \ref{prop:few_high_fluct_boxes}]
    For $\alpha \in (0,1)$, $\eta > 0$, and $T \subset \Lambda_{n^\alpha/2}$, define 
    \begin{align*}
        A_{T, \eta}(B) \coloneqq \Big\{\sup_{y \in T}|\phi^B({x_B + y})| > \eta \log n \Big\}.
    \end{align*}
    Note that by the domain Markov property of GFF (see Lemma \ref{lemma:domainMP}), under $\P_n$, the events $(A_{T, \eta}(B))_{B \in \Pi_\alpha}$ are independent and have the same probability. To estimate the latter, observe that for any $y \in T$ and $B \in \Pi_\alpha$, 
    $$\Var_n\big[(\phi^B({x_B + y}))^1 \big] = G_B (x_B + y, x_B + y) = G_{\Lambda_{n^\alpha}}(y,y) \in g \log (n^\alpha) + (- c, c)$$ for an appropriate absolute constant $c > 0$ (see \eqref{eq:Greensfct_bounds}). Hence, for all $n$ sufficiently large,
    \begin{align*}
        \P_n [A_{T, \eta}(B) ] &\leq |T| \sup_{y \in T} \P_n\big[|\phi^B({x_B + y})| > \eta \log n\big] \\
        &\leq N|T| \sup_{y \in T} \P_n\bigg[|(\phi^B({x_B + y}))^1| > \frac{\eta \log n}{\sqrt{N}} \bigg]\\
        &\leq 2N |T| \P\bigg[X \geq \sqrt{\frac{\eta^2\log n}{2 \alpha N}} \bigg] = \mathcal{O}\Big(n^{-\frac{\eta^2}{8 \alpha N}}\Big),
    \end{align*}
    where $X$ is a standard normal variable.

    Now that we have $M^\alpha_{T}(\eta) = \sum_{B \in \Pi_\alpha} \ind_{A_{T, \eta}(B)}$ written as a sum of i.i.d. bounded random variables, we want to use Bernstein's inequality. To do so, we first need to center the variables: for $B\in \Pi_\alpha$, define
    \begin{align*}
        Z_B \coloneqq \ind_{A_{T, \eta}(B)} - \P_n [A_{T, \eta}(B)] \, \in [-1, 1].
    \end{align*}
    By the above discussion, we get that
    $$\Var[Z_B] = \P_n [A_{R, \eta}(B)] - \P_n [A_{R, \eta}(B)]^2 = \mathcal{O}\Big(n^{-\frac{\eta^2}{8 \alpha N}}\Big).$$
    Therefore, for all $\eta \leq (12 \alpha N)^{-1}$ (for the first inequality and relation $t \gg \Var[Z_B] |\Pi_\alpha|$ with $t$ defined below to be true), by Bernstein's inequality for independent centered bounded (by $1$) variables, we obtain
    \begin{align*}
        \P_n \big[M^\alpha_T(\eta) \geq n^{-\eta^3} |\Pi_\alpha| \big] &\leq \P\bigg[ \sum_{B \in \Pi_\alpha} Z_B \geq \underbrace{n^{-\eta^3} |\Pi_\alpha| \Big(1- \mathcal{O}\big(n^{-\frac{\eta^3}{2}}\big)\Big)}_{\eqqcolon t}\bigg]\\
        &\leq \exp\bigg({-\frac{t^2}{2t/3 + 2\Var[Z_B] |\Pi_\alpha|}}\bigg) \leq e^{-t}.
    \end{align*}
    But notice that $t \geq n^c$ for any $0<c< 2(1-\alpha) - \eta^3$. This completes the proof with $\eta_0(\alpha, N) \coloneqq (12\alpha N)^{-1}$.
\end{proof}


\section[Freezing of spins]{Freezing of spins of the conditioned field}
\label{sec:freezing}
The goal of this section is to prove freezing of spins $\sigma = \phi/|\phi|$ for the (vector-valued) Dirichlet GFF $\phi$ conditioned to avoid a ball (see Theorem \ref{thm:freezing}). We first address the vector-valued case ($N \geq 2$) in Subsection \ref{subsec:vector.freezing}, where we establish alignment of spins at all mesoscopic scales (Theorem \ref{thm:vector.freezing}). This result is weaker than its scalar counterpart and follows from the ingredients developed so far.
In Subsection \ref{subsec:N=1freezing}, we turn to the scalar case ($N = 1$), where we prove that spins are completely aligned at all scales, except on a negligible fraction of sites (Theorem \ref{thm:N=1freezing}), and obtain an upper bound for the probability of the conditioning event (Proposition \ref{prop:N=1upperbound_condition}).

\subsection{Vectorial case: Freezing at mesoscopic scales}
\label{subsec:vector.freezing}
Throughout this subsection let $N\geq 2$, $\phi$ be an $N$-vectorial Dirichlet GFF on $\Lambda_n$, and $I$ be an open ball $\B^N(0,R) \subset \R^N$ for some $R>0$. The objective of this subsection is to prove that spins of the conditioned field freeze at mesoscopic scales.
\begin{theorem}[Freezing of spins of the conditioned field]
    \label{thm:vector.freezing}
    For every $\nu \in (0,1)$ and for any $\varepsilon > 0$, there exists $\tilde D_n \subset D_n$ such that $|D_n \setminus \tilde D_n| = o(n^2)$ and for all $x,y \in \tilde D_n$ with $|x-y| \leq n^\nu$, 
    \begin{align}  \label{eq:freezing_pairs}
        \E_n\bigg[ \frac{\phi(x) }{|\phi(x)| } \cdot \frac{\phi(y)}{ |\phi(y)|}~\bigg|~\Omega_{D_n}^I \bigg] \geq 1 - \varepsilon
    \end{align}
    for all $n$ sufficiently large.
\end{theorem}
\begin{remark}
    Unfortunately, we are unable to prove this theorem at the macroscopic scale, although we still expect the statement to hold in that setting. The main difficulty lies in showing that the harmonic extension of the conditioned field inside a set of macroscopic size has a norm of order at least $\varepsilon \log n$. Our current argument does not apply in this case, since it relies on the fact that we can partition the space into a number of connected components that diverges as $n \to \infty$.
\end{remark}

The idea behind the proof is to consider $n^\nu$-shifts of the $\alpha$-grid $\Pi_\alpha(0)$ for $\alpha > \nu$ sufficiently close to $1$. This allows us to define a random set of good sites, characterized so that, for each of them, the spin at the site is nearly aligned with the spin at the nearest center $x_B$ (within distance less than $n^\nu$) of some $\alpha$-box $B$ in one of the shifted grids. The passage to a deterministic set of good sites is then carried out analogously to the proof of entropic repulsion. 
\begin{proof}
    Let $\beta \in (0, 1/2)$ be small (depending on $\varepsilon$), $\alpha \in (\nu,1)$, and $\eta \in (0, \beta]$—all to be specified later. First, consider the ``reference" set of centerings (where $n^\nu/2$ denotes the integer part of $n^\nu/2$)
    \begin{align*}
        G \coloneqq \frac{1}{2} n^\nu \Z^2 \cap \Lambda_{n^\alpha}, 
    \end{align*}
    and define the random set of good sites in $D_n$ by $\G_{D_n}^\alpha(\beta, \eta) = \{w \in D_n: A_w\}$ with the event
    \begin{align*}
         A_w = \bigg\{\frac{|h_B|}{\log n} \geq 2 - \beta\, \text{ and } \,\frac{|\phi^B({x_B})| \vee |\phi^B({w})|}{\log n} \leq \eta, \, \forall B \in \bigcup_{x_0 \in G} \Pi_\alpha(x_0) \text{ with } w-x_B \in \Lambda_{4n^\nu} \bigg\}.
    \end{align*}
    Note that 
    \begin{align*}
        |\G_{D_n}^\alpha(\beta, \eta)| \geq |D_n| - \sum_{x_0 \in G} \sum_{w \in \Lambda_{4n^\nu}} \big( N^\alpha(\beta, x_0) +  M_{\{0, w\}}^\alpha(\eta, x_0)\big).
    \end{align*}
    Therefore, by \eqref{eq:lowerbound_condevent} and Propositions \ref{prop:few_low_boxes} and \ref{prop:few_high_fluct_boxes}, for $\alpha \in (\nu,1)$ sufficiently close to $1$ (depending on $\beta$), $\eta >0$ sufficiently close to $0$ (depending on $\alpha$)—for example, choose $\alpha \coloneqq \alpha_0(\beta) \vee \frac{1+\nu}{2}$ and then $\eta \coloneqq \eta_0(\alpha, N) \wedge \beta \wedge (1-\alpha)^{1/3}$—for all $n$ sufficiently large
    \begin{align*}
        \sup_{x_0 \in G, w \in \Lambda_{n^\nu}} \E_n &\big[N^\alpha(\beta, x_0) +  M_{\{0, w\}}^\alpha(\eta, x_0) ~\big|~ \Omega^I_{D_n}\big] - 2n^{2(1-\alpha) - \eta^3} \\
        &\leq  n^{2(1-\alpha)} \sup_{x_0 \in G, w \in \Lambda_{n^\nu}}\P_n \big[N^\alpha(\beta, x_0) \vee M_{\{0, w\}}^\alpha(\eta, x_0) \geq n^{2(1-\alpha) - \eta^3} ~\big|~ \Omega^I_{D_n}\big]\\
        &\leq 2n^{2(1-\alpha)} e^{-\frac{1}{2} \log^2 n}.
    \end{align*}
    This implies that 
    \begin{align*}
        \E_n \big[\big|\G_{D_n}^\alpha(\beta, \eta)|~\big|~ \Omega^I_{D_n}\big] \geq |D_n| - 128(1+o(1)) n^{2(\alpha - \nu)} n^{2\nu} n^{2(1-\alpha) - \eta^3} = |D_n| \big(1-\mathcal{O}(n^{-\eta^3})\big),
    \end{align*}
    so fully analogously to the proof of Theorem \ref{thm:repulsion_ball}, we obtain that the set of sites
    \begin{align*}
        \tilde D_n \coloneqq \Big\{w \in D_n: \P_n\big[A_w ~\big|~\Omega^I_{D_n}\big] \geq 1 - \frac{\varepsilon^2}{32}\Big\}
    \end{align*}
    contains at least $|D_n|(1- n^{-\eta^3/2})$ vertices for all $n$ sufficiently large.

    To complete the proof, it remains to show that for any $x, y \in \tilde D_n$ with $|x-y| < n^\nu$, \eqref{eq:freezing_pairs} holds.
    To this end, write $\sigma(x) = \phi(x)/|\phi(x)|$ for brevity and let $x_0 \in G + n^\alpha \Z^2$ be the nearest site to $x$ (choose one arbitrarily if there are multiple (at most four) such vertices)—in particular, $x, y \in x_0 + \Lambda_{4n^\nu}$. Then, by Cauchy-Schwarz inequality,
    \begin{align*}
         \E_n\big[ \sigma(x) \cdot \sigma(y)~\big|~\Omega_{D_n}^I \big] &\geq 1 - \sum_{z=x, y} \sqrt{\E_n\big[ |\sigma(z) - \sigma(x_0)|^2~\big|~\Omega_{D_n}^I \big]} \\
         &= 1- \sum_{z=x, y} \sqrt{2 - 2\E_n\big[ \sigma(z) \cdot \sigma(x_0)~\big|~\Omega_{D_n}^I \big]} \geq 1 - \varepsilon
    \end{align*}
    as desired if $\E_n[ \sigma(z) \cdot \sigma(x_0)~|~\Omega_{D_n}^I ] \geq 1 - \varepsilon^2/8$ for $z = x, y$. To prove the latter, let $B$ denote the $\alpha$-box centered at $x_0$, i.e., $B \in \Pi_\alpha(x_0)$ with $x_B = x_0$, and recall that by \eqref{eq:harmdiff_bound}, for an appropriate absolute constant $c>0$ and all $z \in D_n$ with $|z-x_B| \leq 4 n^{\nu}$, $$\Var_n[(h^{\primorial_\alpha(x_0)}(z) - h_B)^1] \leq c n^{2(\nu-\alpha)} = o\Big( \frac{1}{\log^2 n}\Big).$$ Combined with \eqref{eq:lowerbound_condevent}, this yields that the event
    \begin{align*}
        F \coloneqq \Big\{|h^{\primorial_\alpha(x_0)}(z) - h_B| \geq 1,\, \forall z \in \{x, y\} \Big\}
    \end{align*}
    has vanishing (in the limit $n\to \infty$) probability under $\P_n[\cdot|\Omega^I_{D_n}]$—see proof of Proposition \ref{prop:few_low_boxes} for a similar argument.
    Furthermore, notice that on the event $A_z \cap F^c$,
    \begin{align*}
        \sigma(z) \cdot \sigma(x_0) \geq  \frac{(2 - \beta)^2}{(2 - \beta + \eta )^2}- \frac{\eta^2 + 2\eta(2  - \beta)}{(2 - \beta - \eta)^2} - o(1)
    \end{align*}
    Recalling that $\eta \leq \beta$, and choosing $\beta>0$ sufficiently small so that the right hand-side of the last expression is greater than or equal to $1- \varepsilon^2/32$, we may conclude. Indeed, since then, by definition of $\tilde D_n$,
    \begin{align*}
        \E_n\big[ \sigma(z) \cdot \sigma(x_0)~\big|~\Omega_{D_n}^I \big] &\geq 1- \frac{\varepsilon^2}{32} - 2 \Big(\P_n\big[ A_z^c~\big|~\Omega_{D_n}^I \big] + \P_n\big[ F~\big|~\Omega_{D_n}^I \big]\Big) \geq 1-\frac{\varepsilon^2}{8}
    \end{align*}
    for all $n$ sufficiently large.
\end{proof}

\subsection{One spin-dimensional case: Freezing at all scales}
\label{subsec:N=1freezing}
Throughout this subsection, we work with $N=1$. The main result we prove is the freezing of spins at all (including macroscopic) distances.
\begin{theorem}[Freezing of signs of the conditioned field]
\label{thm:N=1freezing}
    Let $I\subseteq \R$ be a proper interval. For any $\varepsilon > 0$, there exists $\tilde D_n \subset D_n$ such that $|D_n \setminus \tilde D_n| = o(n^2)$ and for all $x,y \in \tilde D_n$, 
    \begin{align*}
        \E_n\bigg[\frac{\phi(x)}{|\phi(x)|} \cdot \frac{\phi(y)}{|\phi(y)|} ~\bigg|~\Omega_{D_n}^I\bigg] \geq 1 - \varepsilon
    \end{align*}
    for all $n$ sufficiently large.
\end{theorem}
First note that if $I= (-\infty, b)$ or $(b,\infty)$ for $b\in \R$ the statement of the theorem with $\tilde D_n= D_n$ is an immediate consequence of \eqref{eq:1sD_norm}. So, assume from now on that $I = (a,b)$ for some $a<b \in \R$. The two main results that constitute the proof are the entropic repulsion of the norm as in Section \ref{sec:repulsion} and Proposition \ref{prop:pick_sign} just below that informally states that the centers of the field pick a sign under the conditioning. To be more precise, under $\Omega^I_{D_n}$, the sign of $h_B$ remains constant in almost all boxes of $\Pi_\alpha$ (for any fixed $\alpha \in (0,1)$ and all $x_0 \in \Z^2$). 
\begin{proposition}[Centers pick a sign]
\label{prop:pick_sign} 
    There exists an absolute constant $c_0 > 0$ such that for $\alpha \in (0,1)$, $x_0 \in \Z^2$, and $u \in (0, 2(1-\alpha))$,
    \begin{align*}
        \P_n \big[N^\alpha_+(x_0) \wedge N^\alpha_-(x_0) \geq n^u ~\big|~\Omega^I_{D_n}\big] \leq e^{-c_0 \alpha^2 \log^2 n}
    \end{align*}
    for all $n$ sufficiently large. Here,
    \begin{align*}
        N^\alpha_\pm(x_0) \coloneqq \# \big\{B \in \Pi_\alpha(x_0):  \pm h_B > 0 \big\}.
    \end{align*}
\end{proposition}

We defer the proof of Proposition \ref{prop:pick_sign} to Subsubsection \ref{subsubsec:pick_sign} and now proceed with the proof of Theorem \ref{thm:N=1freezing}, assuming the validity of the proposition. We begin by introducing notation analogous to $N^\alpha$ and $M^\alpha$—which involve $h^{\primorial_\alpha}$ and $\phi^B$, respectively—this time referring directly to the field $\phi$:
\begin{align*}
    L^\alpha_{\pm}(x_0) &\coloneqq \#\big \{B\in \Pi_\alpha(x_0): \pm \phi({x_B})>0 \big \},\\
    L^\alpha(\beta, x_0) &\coloneqq \#\big \{B\in \Pi_\alpha(x_0): |\phi({x_B})|< (2 - \beta)\log n \big \}.
\end{align*}

\begin{proof}[Proof of Theorem \ref{thm:N=1freezing}]
Observe that for all $\beta \in (0,1)$, $\alpha \in (0,1)$, $0<\eta \leq \beta/2$, and $x_0 \in \Z^2$,
    \begin{align*}
        L^\alpha(\beta, x_0) &\leq N^\alpha(\beta/2, x_0) + M^\alpha(\eta, x_0);\\
        L^\alpha_{+}(x_0) \wedge L^\alpha_-(x_0) &\leq N^\alpha(\beta, x_0) + M^\alpha(\eta, x_0) + N^\alpha_+(x_0) \wedge N^\alpha_-(x_0). 
    \end{align*}
    So, let $\beta \in (0, 1)$ be arbitrary but fixed. Then, by \eqref{eq:lowerbound_condevent} together with Propositions \ref{prop:few_low_boxes}, \ref{prop:few_high_fluct_boxes}, and \ref{prop:pick_sign} we get that for all $\alpha$ close enough to $1$ and $\eta>0$ small enough, there exists $c= c(\beta, \alpha, \eta) > 0$ such that 
    \begin{align}
    \label{eq:bad_centers_negligible}
        \P_n\big[ L^\alpha(\beta, x_0), \, L^\alpha_{+}(x_0) \wedge L^\alpha_{-}(x_0) \leq n^{2(1-\alpha) - \eta^3} ~\big|~\Omega^I_{D_n}\big] \geq 1 - e^{- c\log^2 n}
    \end{align}
    for all $x_0 \in \Z^2$ and for all $n$ sufficiently large. To reduce the number of parameters, we set, for example, $\alpha \coloneqq \alpha_0(\beta/2) \vee \frac{1}{2}$ (cf. \eqref{eq:alpha_choice}) and $\eta \coloneqq \frac{\beta}{2} \wedge \eta_0(\alpha, 1) \wedge (1-\alpha)^{1/3}$. In this case, $c$ can be set to $\frac{c_0}{8} \wedge\frac{1}{4}$ with $c_0>0$ as in Proposition \ref{prop:pick_sign}. 

    For each $x_0 \in \Z^2$, we now define the sign associated with the corresponding grid $\primorial_\alpha(x_0)$ by the \textit{popular vote} in the following sense
    \begin{align*}
        \sign(x_0) \coloneqq \sgn\left( L^\alpha_{+}(x_0) - L^\alpha_{-}(x_0) \right).
    \end{align*}
    Note that $\sign(x_B) = \sign(x_0)$ for all $B \in \Pi_\alpha(x_0)$. We also introduce the associated random set of \emph{good signed} boxes—boxes whose centers \textit{voted} for the sign of the grid and where the norm of the field is large:
    \begin{align*}
        \G^\alpha_{\sign}(x_0, \beta)\coloneqq \big\{x_B: B\in \Pi_\alpha(x_0), \,|\phi({x_B)}| \geq (2 - \beta) \log n,\, \sgn(\phi({x_B})) = \sign(x_0) \big\}.
    \end{align*}
    We are interested in having a \textit{landslide victory}. To this end, note that 
    \begin{align*}
        |\G_{\sign}^\alpha(x_0, \beta)| \geq |\Pi_\alpha(x_0)| - L^\alpha(\beta, x_0) - L_{+}^\alpha(x_0) \wedge L^{\alpha}_-(x_0)
    \end{align*}
    Then, in light of \eqref{eq:bad_centers_negligible}, fully analogously to the proof of Theorem \ref{thm:repulsion_ball}, we indeed see that with high probability the election for each grid is won by a landslide, and furthermore, we obtain\footnote{Using auxiliary disjoint sets $\Cc_{\sign}^\alpha(x_0, \beta, \varepsilon)$ for $x_0 \in \Lambda_{n^\alpha}$ (up to the inner boundary) given by $\{x_B: B \in \Pi_\alpha(x_0), \, \P_n\big[ x_B \in \G^\alpha_{\sign}(x_0, \beta)~\big|~\Omega^I_{D_n}\big] \geq 1 - \frac{\varepsilon}{8}\}$} that for all $n$ sufficiently large, the deterministic set of good signed boxes defined as
    \begin{align*}
        \tilde D_n \coloneqq \Big\{z \in D_n: \, \P_n\big[z \in \G^\alpha_{\sign}(z, \beta) ~\big|~\Omega^I_{D_n}\big] \geq 1 - \frac{\varepsilon}{8}\Big\}
    \end{align*}
    is of cardinality $|D_n| (1-o(1))$.

    To complete the proof of the theorem, it suffices to show that with high probability, the random variable $\sign(x_0)$ does not depend on $x_0 \in \Lambda_{n^\alpha}$:
    \begin{claim} \label{claim:grid_sign_indep}
        There exists an absolute constant $\tilde c>0$ such that for all $x_0\in D_n$, 
        \begin{align*}
            \P_n\big[\sign(x_0) \neq \sign(0) ~\big|~\Omega_{D_n}^I\big] \leq e^{-\tilde c \log^2 n}
        \end{align*}
        as long as $n$ is large enough. 
    \end{claim}
    \noindent Indeed, since then for any $x, y \in \tilde D_n$
    \begin{align*}
        \E_n\bigg[\frac{\phi(x)}{|\phi(x)|} \cdot \frac{\phi(y)}{|\phi(y)|} ~\bigg|~\Omega_{D_n}^I\bigg] &= 1 - 2 \P_n \big[ \sgn(\phi(x)) \neq \sgn(\phi(y))~ |~ \Omega_n^I\big] \\
        &\geq 1- 4 \max_{z =x, y}\P_n\big[ \sign(z)\neq \sign(0) \text{ or } \sgn(\phi(z))\neq \sign(z)~ |~\Omega_{D_n}^I\big] \\
        &\geq 1- \frac{\varepsilon}{2} - 4 e^{-\tilde c \log^2 n} \geq 1 - \varepsilon
    \end{align*}
    for all $n$ sufficiently large as desired. 

    We proceed with the proof of the remaining claim.
\begin{proof}[Proof of Claim \ref{claim:grid_sign_indep}]
    Note that it suffices to consider $x_0\in \Lambda_{n^\alpha}$. We first verify the claim for sites $x_0 \in \Lambda_{n^\alpha/K}$ for $K \geq 8$ large, to be specified later. In what follows, $\beta\in(0,1)$ and $\alpha, \eta$ in terms of it are as above (in particular, $\eta \leq \beta/2$). Observe that by \eqref{eq:bad_centers_negligible} and Proposition \ref{prop:few_high_fluct_boxes}, the probability (under $\P_n [\cdot |\Omega^I_{D_n}]$) of the complement of the event 
    \begin{align*}
        A(0,x_0) \coloneqq \Big\{L^\alpha_+(z)\wedge L^\alpha_-(z), \, L^\alpha(\beta, z), \, M^\alpha_{\{0, x_0\}}(\eta, 0)\leq n^{2(1-\alpha)-\eta^3} \text{ for both } z=0, x_0\Big\}
    \end{align*}
    is bounded by $2e^{-c \log^2 n}$ for all $n$ sufficiently large depending on $\beta$. Here, $c>0$ is an absolute constant as above. Furthermore, notice that on the event $\{\sign(x_0) \neq \sign(0)\} \cap A(0,x_0)^c$, the set of boxes
    \begin{align*}
        \bigg\{B \in \Pi_\alpha(0): \frac{ |h_B-h^{\primorial_\alpha(0)}({x_B+x_0})|}{\log n}\geq 4 - 2\beta - 2\eta \bigg\}
    \end{align*}
    has cardinality at least $|\Pi_\alpha(0)|(1 - \mathcal{O}(n^{-\eta^3}))$, which in turn clearly implies that there exists at least one such box in $\Pi_\alpha(0)$.
    Since by \eqref{eq:harmdiff_bound}, $$\Var_n[h_B-h^{\primorial_\alpha(0)}({x_B+x_0})] \leq C/K^2$$ for some absolute constant $C>0$ and all $x_0 \in \Lambda_{n^\alpha/K}$, the (unconditional) probability of the latter event is bounded by
    \begin{align*}
        n^{2(1-\alpha)} \sup_{B \in \Pi_\alpha(0)}\P_n \big[|h_B-h^{\primorial_\alpha(0)}({x_B+x_0})|\geq (4 - 3\beta) \log n \big] \leq n^{2(1-\alpha)} e^{-\frac{(4 - 3\beta)^2 K^2}{2C} \log^2 n}
    \end{align*}
    Here we also used that $\eta \leq \beta/2$. By choosing $K$ (in dependence on $\beta$ and $D$) large enough such that $\frac{(4 - 3\beta)^2 K^2}{2C} \geq \frac{2}{\pi} \cpc_\Lambda(D) + 2$, in view of the above and \eqref{eq:lowerbound_condevent}, the proof in the case $x_0 \in \Lambda_{n^\alpha/K}$ is complete.

    To obtain the desired result for a general vertex $x_0 \in \Lambda_{n^\alpha}$, observe that the proof of the previous case translates (with all parameters and constants preserved) directly to the case of an arbitrary fixed vertex $z_*\in \Lambda_{n^\alpha}$ (not necessarily $0$) and all sites $z \in \Lambda_{n^\alpha/K} + z_*$ close to it, so define the finite ``reference" set 
    \begin{align*}
        G \coloneqq \frac{n^\alpha}{4K} \Z^2 \cap \Lambda_{n^\alpha}.
    \end{align*}
    Now, for any $x_0 \in \Lambda_{n^\alpha}$, there exists a sequence $0 = z_1, z_2, \dots, z_l \in G$ ($l \leq |G|$) such that $z_i \in G$, $|z_{i+1} - z_i| = \frac{n^\alpha}{4K}$, and $x_0 \in \Lambda_{n^\alpha/K} + z_l$, which by the above implies that for all $n$ sufficiently large
    \begin{align*}
        \P_n\big[\sign(x_0) \neq \sign(0) ~\big|~\Omega_{D_n}^I\big] \leq \sum_{i=1}^{l} \P_n\big[\sign(z_i) \neq \sign(z_{i+1}) ~\big|~\Omega_{D_n}^I\big] \leq |G| e^{-\tilde c \log^2 n}.
    \end{align*}
    Here for brevity we write $z_{l+1} = x_0$. This finishes the proof of the claim.
\end{proof}
    The proof of the theorem is now complete.
\end{proof}
 
Before moving on to proving Proposition \ref{prop:pick_sign}, we show that combined with the other two key ingredients (Propositions \ref{prop:few_low_boxes} and \ref{prop:few_high_fluct_boxes}), it gives us the missing upper bound on the probability of $\Omega^I_{D_n}$ such that \eqref{eq:lowerbound_condevent} becomes an equality for $N=1$. Our proof follows closely the one of \cite[Lemma 10]{entropRep2D}.
\begin{proposition}[Upper bound on the probability of the conditional event]\label{prop:N=1upperbound_condition}
    We have the following:
    \begin{align}
    \label{eq:prob_cond.event}
        \lim_{n \rightarrow \infty} \frac{\log \P_n[ \Omega_{D_n}^I]}{\log^2 n} \leq -\frac{2}{\pi} \:\cpc_\Lambda(D).
    \end{align}
\end{proposition}
\begin{proof}
    Let $x_0 = 0$. For brevity, we suppress the dependence on $x_0 = 0$ and write $N^\alpha_\pm$ and $N^\alpha(\beta)$ instead of $N^\alpha_\pm(0)$ and $N^\alpha(\beta,0)$, respectively. By Propositions \ref{prop:pick_sign} and \ref{prop:few_low_boxes} and symmetry, it suffices to prove that for $\beta \in (0, 2)$, $\alpha \in (0,1)$ sufficiently close to $1$ depending on $\beta$, and $u \in (0,2(1-\alpha))$, 
    \begin{align}
    \label{eq:remains_to_bound}
        \limsup_{n \rightarrow \infty} \frac{\log \P_n[\Omega^I_{D_n} \cap \{N^\alpha_- , N^\alpha(\beta) \leq n^u\}]}{\log^2 n} \leq - \frac{(2 - \beta)^2 }{2\pi} \cpc_\Lambda(D).
    \end{align}
    Let $f \in C^1(\Lambda)$ be non-negative in $D$ (such that $f\vert_D \not\equiv 0$). On the event $\Omega^I_{D_n} \cap \{N^\alpha_- , N^\alpha(\beta) \leq n^u\}$, we have 
    \begin{align*}
        I_f &\coloneqq \frac{1}{n^{2(1-\alpha)}} \sum_{B \in \Pi_\alpha} f\Big(\frac{x_B}{n}\Big) h_B \\
        &\:\geq (2 - \beta) \log n \bigg( \frac{1}{n^{2(1-\alpha)}} \sum_{B \in \Pi_\alpha} f\Big(\frac{x_B}{n}\Big) - \frac{n^u \norm{f}_{L^\infty(D)}}{n^{2(1-\alpha)}} \bigg(2 + \sup_{x \in \primorial_\alpha}\frac{|\phi(x)|}{(2 - \beta)\log n}\bigg) \bigg), 
    \end{align*}
    using the bound $|h_B| \leq \sup_{x \in \primorial_\alpha} |\phi(x)|$ for all $B \in \Pi_\alpha$. Note that
    \begin{align*}
        \P_n\Big[\sup_{x \in \primorial_\alpha} |\phi(x)| \geq \log^2 n\Big] \leq 2n^2 \sup_{x \in \primorial_\alpha} \P_n \big[\phi(x) \geq \log^2 n\big] = \mathcal{O}\big(e^{-c\log^3 n}\big)
    \end{align*} 
    for an appropriate $c>0$ by \eqref{eq:Greensfct_bounds}, making this event negligible compared to $\Omega^I_{D_n}$. On its complement—$\sup_{x \in \primorial_\alpha} |\phi(x)| \leq \log^2 n$—by the above, we obtain that 
    \begin{align*}
        I_f &\geq (2 - \beta) \log n \bigg( \frac{1}{n^{2(1-\alpha)}} \sum_{B \in \Pi_\alpha} f\Big(\frac{x_B}{n}\Big) - \mathcal{O}\big(n^{u - 2(1-\alpha)} \log n\big)\bigg)\\
        &\geq (2 - \beta) (1-o(1)) \log n \int_D f \d x. 
    \end{align*}

    By \cite[Lemma 13]{entropRep2D}, $$\lim_{n\rightarrow \infty} \frac{1}{2\pi} \Var_n[I_f] = \sigma^2_\Lambda(f \ind_D) = \int_D\int_D f(x) G_\Lambda(x,y) f(y) \d x \d y,$$
    where $G_\Lambda$ is the Green's function\footnote{Here, the Brownian motion is analytically normalized, in contrast to the reference, which employs the random walk (RW) normalization: $G_\Lambda(x,y)$ here equals $\frac{1}{2d}$ times theirs, which offsets the typo in their definition of capacity. The prefactor $1/(2\pi)$ compensates for our choice of normalization of the GFF.} of the Brownian motion killed as it exits $\Lambda$. 
    In view of this and the preceding steps, we have
    \begin{align*}
        \limsup_{n \rightarrow \infty} \frac{\log \P_n[ \Omega_{D_n}^I]}{\log^2 n} \leq -\frac{(2 - \beta)^2}{2\pi} \frac{(\int_D f \d x)^2}{2 \sigma_\Lambda^2(f \ind_D)}.
    \end{align*}
    We conclude the proof by optimizing in $f$ and recalling that the relative capacity can also be expressed (cf.~\cite[Proof of Theorem 2.2]{Deuschel1996} combined with \cite[Lemma 2.1]{critLDP} and \cite[(A.6)]{entropRep3+D}) as
    \begin{align*}
        \cpc_\Lambda(D) = \sup\bigg\{\frac{(\int_D f\d x)^2}{2\sigma_\Lambda^2(f \ind_D)}: f \in C^1(\Lambda), f\vert_D \not\equiv 0\bigg\},
    \end{align*}
    see also \cite[Exercise 2.2.2 and Example 2.2.2]{FukushimaOshimaTakeda0} for the reduction to non-negative functions.
\end{proof}

\subsubsection{Centers pick a sign}
\label{subsubsec:pick_sign}

In this subsection, we finally prove Proposition \ref{prop:pick_sign}, the last missing part of Theorem \ref{thm:freezing}. The main idea behind the proof of Proposition \ref{prop:pick_sign} is the observation that if the centers do not pick a sign, we are in one of the two possible scenarios: either there are many centers, where $h_B$ is close to $0$, or there is a big interface between centers that have picked different signs with jumps $|h_B - h_{B'}|$ of noticeable magnitude. The former can be discarded by Proposition \ref{prop:few_low_boxes} (in view of \eqref{eq:alpha_choice}).   
The latter situation can also be ruled out, by exploiting the fact that by adding an appropriate independent Gaussian field, one transforms the field of differences of $h_B$ into a white noise, and then by applying Bernstein's inequality.
\begin{proof}[Proof of Proposition \ref{prop:pick_sign}]
    Note that it suffices to show that there exists an absolute constant $c_0 > 0$ such that for $\alpha \in (0,1)$, $x_0 \in \Z^2$, and $u \in (0, 2(1-\alpha))$, 
    \begin{align*}
        \P_n \big[N^\alpha_+(x_0) \geq N^\alpha_-(x_0) \geq n^u ~\big|~\Omega^I_{D_n}\big] \leq e^{-c_0 \alpha^2 \log^2 n},
    \end{align*}
    as the remaining case follows by considering $I = (-b, -a)$ and symmetry. In the sequel, we suppress dependence on $x_0$ for better readability. To prove this, we partition $\Pi_\alpha$ into two subsets: we say that $B \in \Pi_\alpha$ is either \emph{positive} or \emph{negative}, according to the sign of $h_B$. 
    This allows us to define $\II_\alpha \subset \,\primorial_\alpha$, the random interface between the negative and positive boxes contained in $D_n$ (see Figure \ref{fig:interfaces}), i.e., the collection of inner boundaries of boxes of $\Pi_\alpha$ that are shared by both a positive and a negative box.
    \begin{figure}[ht]
        \centering
        \includegraphics[width=0.7\linewidth]{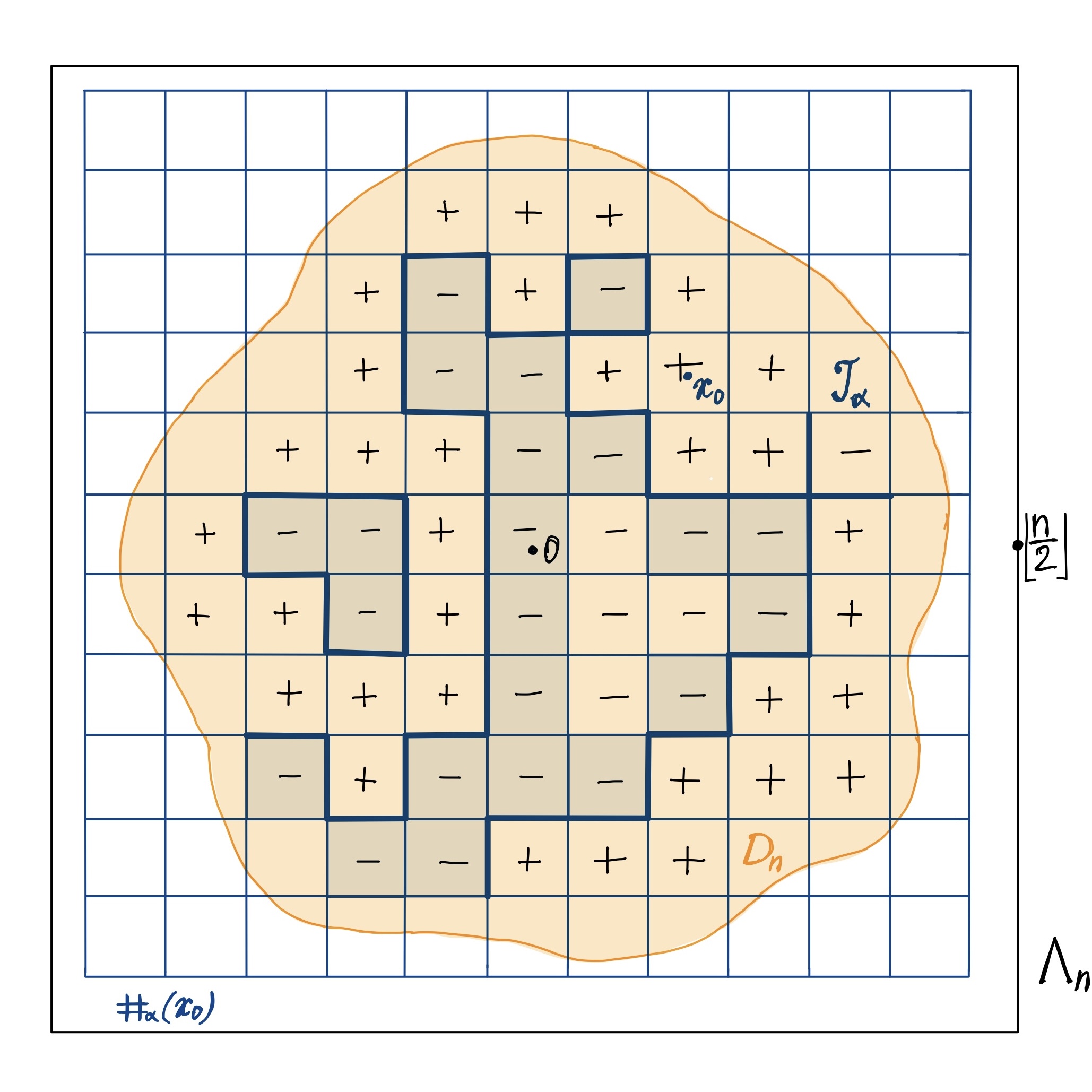}
        \caption{The part of the \emph{blue} subgrid of $\Z^2$ contained in $D_n$ is $\primorial_\alpha(x_0)$. The union of boxes $B$ formed by this subgrid is $\Pi_\alpha(x_0)$. The thick \emph{blue} lines form $\II_\alpha$, the interface between \emph{positive} and \emph{negative} boxes contained in $D_n$. The \emph{blue} boxes are exactly the ones counted in $N^\alpha_- |_{\II_\alpha}= N^\alpha_{<-\delta}|_{\II_\alpha} + N^\alpha_{-\delta, 0}|_{\II_\alpha}$.}
        \label{fig:interfaces}
    \end{figure} 

    Let $\delta > 0$ be small enough (to be specified later). We separate negative boxes into two types: boxes with $\frac{h_B}{ \log n}$ ``sufficiently negative" and ``close to zero", respectively; and introduce the following notation for the cardinality of the corresponding sets
    \begin{align*}
        N^\alpha_{<-\delta} &\coloneqq \#\big\{B \in \Pi_\alpha: h_B < -\delta  \log n\big\}; \\
        N^\alpha_{-\delta, 0} &\coloneqq \#\big\{B \in \Pi_\alpha \cap  D_n: -\delta  \log n \leq h_B < 0\big\}
    \end{align*}  
    Furthermore, for $N \in \{N^\alpha_-, N^\alpha_{<-\delta}, N^\alpha_{-\delta, 0}\}$, we denote by $N |_{\II_\alpha}$ the number of boxes of the respective type, which have at least one side of the inner boundary in $\II_\alpha$.    

    By the isoperimetric inequality, there exists an absolute constant $c>0$ such that on the event $\{N^\alpha_+ \geq N^\alpha_- \geq n^u\}$, $N^\alpha_{<-\delta}|_{\II_\alpha} \geq c n^{u/2}$ or $N^\alpha_{-\delta, 0}|_{\II_\alpha} \geq c n^{u/2}$. In particular, we get 
    \begin{align*}
        \P_n \big[\Omega^I_{D_n} \cap \{N^\alpha_+ \geq N^\alpha_- \geq n^u\}\big] \leq &\P_n\big[\{N^\alpha_{-\delta, 0}|_{\II_\alpha} \geq c n^{u/2}\} \cap \Omega^I_{D_n}\big] \\
        + &\P_n \big[\{N^\alpha_{<-\delta}|_{\II_\alpha} \geq c n^{u/2}\} \cap \Omega^I_{D_n} \big] \eqqcolon p_1 + p_2.
    \end{align*}
    Observe that by Proposition \ref{prop:few_low_boxes} and its proof (more precisely, \eqref{eq:alpha_choice}), the first summand $p_1$ is bounded by $e^{-(\frac{2}{\pi}\cpc_\Lambda(D) + 1)\log^2 n}$ as long as $\alpha \geq \alpha_0(2 - \delta) = \frac{\delta}{1 + \delta/2}$, which is verified for $\delta \coloneqq \alpha$.  Combined with \eqref{eq:lowerbound_condevent}, this yields that for all $n$ sufficiently large, $p_1/\P_n[\Omega^I_{D_n}] \leq e^{-\frac{1}{2}\log^2 n}$. 

    We proceed to estimating the second summand $p_2$. To this end, we consider the Gaussian process $Z$ defined by $Z({B,p})=h^{\primorial_\alpha}_{B+ n^\alpha p}-h_B$ for $B\in \Pi_\alpha$ and $p=\{1,i\} \subset \Z^2 \subset \C$. Note that on the event $\{N^\alpha_{<-\delta}|_{\II_\alpha} >cn^{u/2}\}$, the number of pairs $(B,p)$ such that $|Z({B,p})|\geq \delta \log n$ is at least $c n^{u/2}$. We now show that this is extremely unlikely. Note that by Lemma \ref{lemma:harmonic_cov_EVs}, the largest eigenvalue of the covariance matrix $\tilde \Sigma$ of $Z$ is uniformly bounded by $\tilde c>0$. Therefore, $\check{\Sigma} \coloneqq 2\tilde c \mathrm{Id} - \tilde \Sigma$ is positive definite. Define, independently of $Z$ (and $\phi$ in general), $\check Z = (\check Z({B,p}))_{B \in \Pi_\alpha, p\in \{1, \iu\}}$ as a centered Gaussian process with covariance matrix $\check \Sigma$. In particular, $W \coloneqq Z - \check{Z}$ is a white noise, that is, a collection of i.i.d. centered random variables of variance $2\tilde c$. Combined, the above implies that for all $n$ sufficiently large
    \begin{align*}
        \P_n &\big[\Omega^I_{D_n} \cap\{N^\alpha_{<-\delta}|_{\II_\alpha} \geq c n^{u/2} \} \big] \\
        &\leq\P_n\big[\Omega^I_{D_n}\big] \P_n\bigg[\sup_{B, p} |\check Z({B,p})| > \frac{\delta}{2} \log n \bigg] 
        + \P_n \bigg[\sum_{B, p} \ind_{\big\{|W({B,p})| > \frac{\delta}{2} \log n\big\}} \geq c n^{u/2}\bigg]\\
        &\leq \P_n\big[\Omega^I_{D_n}\big] e^{-\hat{c} \alpha^2 \log^2 n} 
        + \P_n \bigg[\sum_{B, p} \ind_{\big\{|W({B,p})| > \frac{\delta}{2} \log n\big\}} \geq c n^{u/2}\bigg]
    \end{align*}
    for an appropriate absolute constant $\hat{c}>0$, and where in the last inequality we used the union bound, the fact that $\delta = \alpha$ and that $\Var_n [\check Z({B,p})] \leq 2\tilde c$.

    It now only remains to control the event that there are many pairs $(B,p)$ with $|W({B,p})|$ very large, which we do using Bernstein's inequality applied to i.i.d. variables 
    \begin{align*}
        \hat W({B,p}):= \ind_{\big\{|W({B,p})| > \frac{\delta}{2} \log n\big\}}- \P\bigg[|W({B,p})| > \frac{\delta}{2} \log n\bigg] \in [-1,1]. 
    \end{align*}
    Note that since $\P[|W({B,p})|\geq \frac{\delta}{2} \log n] \leq e^{-\hat{c}\alpha^2 \log^2n}$, we obtain for all sufficiently large $n$ that
    \begin{align*}
        \P_n \bigg[\sum_{B, p} \ind_{\big\{|W({B,p})| > \frac{\delta}{2} \log n\big\}} \geq c n^{u/2}\bigg] &\leq \P_n \bigg[\sum_{B, p} \hat W({B,p}) \geq \frac{c}{2} n^{u/2}\bigg] \\
        &\leq e^{-\frac{c}{2} n^{u/2}} \leq e^{-(\frac{2}{\pi} \cpc_\Lambda(D) + 1) \log^2 n}.
    \end{align*}
    The proof is completed by invoking Corollary \ref{cor:halfline_implications}.
\end{proof}


\section[Massive GFF conditioned to avoid a ball]{The massive Gaussian free field conditioned to avoid a ball}
\label{sec:massiveGFF}

As a consequence of Theorem \ref{thm:freezing}, we see that any infinite-volume limit of the spins of a vector-valued massless GFF conditioned to avoid a ball at each site is a constant field. Thus, if we hope to understand the spin $O(N)$ model at low temperature using the GFF, we would need to work with its massive counterpart. Motivated by this, we show that the massive vector-valued GFF on $\Z^d$ conditioned to avoid a ball at each site in a growing box has infinite-volume limits, and then discuss their uniqueness in certain regimes.

We start with the possible infinite-volume limits of a vector-valued massive GFF conditioned to avoid a given open set $I \subset \R^N$ at each site of $\Z^d$. To this end, consider the set of Gibbs states on $(I^c)^{\Z^d}$, denoted $\G(\U^I)$, corresponding to the interaction potential $\U^I = \{U^I_F: \emptyset \neq F \ssubset \Z^d\}$ (see \cite[Chapter 6]{StatMech_book} for an introduction to the topic):
\begin{align*}
    U^I_F \coloneqq \begin{cases}
        \frac{2d+m^2}{2} |\psi(x)|^2 + \infty \ind_{\psi(x) \in I} & F = \{x\};\\
        -\psi(x)\cdot  \psi(y) & F = \{x,y\} \text{ such that } x\sim y;\\
        0 &\text{otherwise}.
    \end{cases}
\end{align*}
Let $\G_1(\U^I)$ be the subset of $\G(\U^I)$ consisting of measures $\mu$ with uniform first moments at each site, i.e., $\mu$ is such that $\sup_{x \in \Z^d} \mu[|\psi(x)|] \leq C < \infty$.
We now restate the existence part of  Theorem \ref{thm:massiveGFF}  using the definition of $\Omega_{\Lambda_n}^I$ in \eqref{eq:Omega_def} in terms of $\psi$.
\begin{proposition}
\label{prop:massiveGFF_existence}
     Let $d\geq 2$, $N \in \N$, $R > 0$, $I = \B^N(0,R) \subseteq \R^N$, and $\P_{m^2}$ denote the law of an $m^2$-massive $N$-vectorial GFF on $\Z^d$. The sequence of probability measures $(\P_{m^2}[\cdot | \Omega^I_{\Lambda_n}])_n$ is tight, and all its subsequential limits lie in $\G_1(\U^I)$. In particular, $\G_1(\U^I)$ is non-empty.
\end{proposition}
In fact, in the scalar case $N=1$, here and in what follows, we consider general intervals of the form $I=(a,b)$ with $-R\leq a<b\leq R$, or $I = (-\infty, b)$.

We now turn to the question of uniqueness, corresponding to the second part of Theorem \ref{thm:massiveGFF}. We show that for sufficiently small $R$ (under an additional constraint on the mass if $N\geq 2$), the measure of interest is unique. In the scalar case, we further study the large-$R$ regime and demonstrate that multiple infinite-volume measures exist in $\G_1(\U^I)$, thereby implying a phase transition for the conditioned scalar massive GFF on $\Z^d$ ($d \geq 2$).
\begin{proposition} \label{prop:massiveGFF_uniqueness} 
    If $m^2 > 2d(N-1)$, then for all $R>0$ sufficiently small (depending on $N, m^2, d$), we have $|\G_1(\U^I)| = 1$; in particular, the sequence $(\P_{m^2}[\cdot | \Omega^I_{\Lambda_n}])_n$ is weakly convergent in this case. For $N=1$, uniqueness also holds for any half-line $I = (-\infty, b)$, whereas for $I = (-R,R)$ with $R>0$ sufficiently large but finite, $|\G_1(\U^I)| > 1$. 
\end{proposition}
Let us mention that for $d=2$ and $N\geq 2$, we expect $\G_1(\U^I)$ to contain only one measure. The same is expected for the spin $O(N)$ model; however, a complete characterization of the corresponding Gibbs measures in infinite volume has not yet been established. 

\begin{remark}
    The proof of Proposition \ref{prop:massiveGFF_uniqueness} for $N=1$  also implies uniqueness in the case $I=(a,b)$ provided one of the endpoints $a$ or $b$ is fixed and the other is sufficiently large in absolute value (depending on $m^2$ and the fixed endpoint).
\end{remark}

\subsection{Tightness and uniform first moments}
\label{subsec:massive_tightness.moment}

This subsection is devoted to the proof of Proposition \ref{prop:massiveGFF_existence}. We begin with the scalar case $N=1$, where $I$ is either $(-\infty, b)$ or $(a,b)$ with $-R \leq a < b \leq R$ for some $R>0$. Let $(\Q_n)_n$ denote either $(\P_{m^2}[\cdot \mid \Omega^I_{\Lambda_n}])_n$ or $(\P_{m^2}[\cdot \mid \Omega^I_{\Lambda_n}, \psi(y) \geq R, \, \forall y \in \Lambda_{2n}\setminus \Lambda_n])_n$. The latter sequence will be used to establish the phase transition in Proposition \ref{prop:massiveGFF_uniqueness}.
\begin{proof}[Proof of Proposition \ref{prop:massiveGFF_existence} for $N=1$]
Observe that it suffices to prove that there exists a constant $C = C(m^2, R, d) > 0$ such that
\begin{align}\label{eq:main_existence_massive}
    \sup_{x \in \Z^d}\sup_n \Q_n[|\psi(x)|] \leq C.
\end{align}
This, in turn, follows from the following claim by choosing $C \geq 2\tilde C$ large enough.
\begin{claim}\label{claim:limsup_moment}
    There exists $\tilde C = \tilde C(m^2, R, d)>0$ such that for any $x \in \Z^d$, 
    \begin{align*}
        \limsup_{n \rightarrow \infty} \E_{m^2}\big[\psi(x) ~\big\vert~ \Omega^{(-\infty,R)}_{\Lambda_n}\big] \leq \tilde C.
    \end{align*}
\end{claim}
Assuming the claim, we now prove \eqref{eq:main_existence_massive}. First note that for $x \in \Lambda_n$, $\E[\psi(x) ~\vert~ \Omega^{(-\infty,R)}_{\Lambda_n}] = \E[(\psi(x))_+ ~\vert~ \Omega^{(-\infty,R)}_{\Lambda_n}]$. Next, by the FKG property (Proposition \ref{prop:FKG_condGFF}) of $\P_{m^2}[\, \cdot \, \vert \Omega^{(-\infty, R)}_{\Lambda_n}]$, the map $n \mapsto \E_{m^2}[(\psi(x))_+ ~\vert~ \Omega^{(-\infty, R)}_{\Lambda_n}]$ is increasing. Together with Claim \ref{claim:limsup_moment}, these two observations imply that
\begin{align*}
    \sup_{x \in \Z^d}\sup_n \E_{m^2}\big [(\psi(x))_+ ~\big\vert~ \Omega^{(-\infty, R)}_{\Lambda_n} \big ] \leq \tilde C.
\end{align*} 

As for \eqref{eq:main_existence_massive}, when $I= (-\infty, b)$, by the FKG property of a massive GFF, since the events $\Omega^I_{\Lambda_n}$ and $ \Omega^I_{\Lambda_n}\cap\, \Omega^{(-\infty,R)}_{\Lambda_{2n}\setminus \Lambda_n}$ are increasing and $\psi \mapsto (\psi(x))_-$ is decreasing, we get that by choosing $C>0$ large enough,
\begin{align*}
    \Q_n[(\psi(x))_- ] \leq \E_{m^2}[(\psi(x))_-] \leq \sqrt{G_{\Z^d, m^2}(0,0)} \leq C/2.
\end{align*}

In the remaining cases ($\Q_n[(\psi(x))_+]$ for any $I$ in question, and $\Q_n[(\psi(x))_-]$ when $a,b$ are finite such that $|a|, |b|\leq R$), we again use the FKG inequality for a conditioned field (Proposition \ref{prop:FKG_condGFF}) and symmetry to obtain
\begin{align*}
    \E_{m^2}\big[(\psi(x))_+ ~\big|~ \Omega^I_{\Lambda_n}\big] &\leq \E_{m^2}\big[(\psi(x))_+ ~\big|~ \Omega^{(-\infty,R)}_{\Lambda_n}\big];\\
    \E_{m^2}\big[(\psi(x))_- ~\big|~ \Omega^I_{\Lambda_n}\big] &\leq \E_{m^2}\big[(\psi(x))_- ~\big|~ \Omega^{(-R,\infty)}_{\Lambda_n}\big] = \E_{m^2}\big[(\psi(x))_+ ~\big|~ \Omega^{(-\infty,R)}_{\Lambda_n}\big].
\end{align*}
Combined wit the above, this proves \eqref{eq:main_existence_massive} for the first definition of $\Q_n$. For the second definition of $\Q_n$, we analogously get
\begin{align*}
    \E_{m^2}\big[(\psi(x))_+ ~\big|~ \Omega^I_{\Lambda_n}, \Omega^{(-\infty,R)}_{\Lambda_{2n}\setminus \Lambda_n}\big] 
    &\leq \E_{m^2}\big[(\psi(x))_+ ~\big|~ \Omega^{(-\infty,R)}_{\Lambda_{2n}}\big];\\
    \E_{m^2}\big[(\psi(x))_- ~\big|~ \Omega^I_{\Lambda_n}, \Omega^{(-\infty,R)}_{\Lambda_{2n}\setminus \Lambda_n}\big] &\leq \E_{m^2}\big[(\psi(x))_- ~\big|~ \Omega^I_{\Lambda_n}, \Omega^{(-R, R)}_{\Lambda_{2n}\setminus \Lambda_n}\big] \\
    &\leq \E_{m^2}\big[(\psi(x))_- ~\big|~ \Omega^{(-R,\infty)}_{\Lambda_{2n}}\big] = \E_{m^2}\big[(\psi(x))_+ ~\big|~ \Omega^{(-\infty,R)}_{\Lambda_{2n}}\big].
\end{align*}
To be precise, to obtain the former pair of inequalities, we used the FKG property of $\P_{m^2} [\cdot \vert \Omega^I_{\Lambda_n}]$ and the facts that $\Omega^{(-\infty, R)}_{\Lambda_n}$ and $\psi \mapsto (\psi(x))_+$ are both increasing and $\Omega^I_{\Lambda_n} \cap \Omega^{(-\infty, R)}_{\Lambda_n} = \Omega^{(-\infty, R)}_{\Lambda_n}$ (similarly for $\Omega^{(-R, \infty)}_{\Lambda_n}$ and $\psi \mapsto (\psi(x))_-$). Analogously, the first inequality of the second pair follows from the FKG of $\P_{m^2}[\cdot | \Omega^I_{\Lambda_n}, \Omega^{(-\infty,R)}_{\Lambda_{2n} \setminus \Lambda_n}]$, while the remaining ones are a consequence of the FKG property of $\P_{m^2}[\cdot | \Omega^I_{\Lambda_n}, \Omega^{(-R,R)}_{\Lambda_{2n} \setminus \Lambda_n}]$.

To complete the proof in the scalar case, it only remains to prove Claim \ref{claim:limsup_moment}.
\begin{proof}[Proof of Claim \ref{claim:limsup_moment}]
    Following the steps of the proof of \cite[Lemma 4.7]{entropRep3+D}, 
    we arrive at
    \begin{align} \label{eq:cond_moment_aux_bound}
         \E_{m^2}\big[\psi(x) ~\big\vert~ \Omega^{(-\infty,R)}_{\Lambda_n}\big] \leq \sqrt{-2 \log \P_{m^2}\big[\Omega^{(-\infty,R)}_{\Lambda_{2n}}\big] ~/~ \big\langle \1_{\Lambda_{n}}, G_{\Z^d, m^2} \vert_{\Lambda_{n}\times \Lambda_{n}}^{-1} \1_{\Lambda_{n}} \big\rangle}.
    \end{align}
    Let $(Z_x)_{x \in \Z^d}$ be i.i.d. centered Gaussian variables of variance $G_{\Z^d, m^2}(0,0)$. Then, by Slepian's lemma, 
    \begin{align*}
        \P_{m^2}\big[\Omega^{(-\infty,R)}_{\Lambda_{2n}}\big] \geq \P[Z_x \geq R, \; \forall x \in \Lambda_{2n}] = \P[Z_0 \geq R]^{|\Lambda_{2n}|}.
    \end{align*}
    And hence,
    \begin{align*}
        \E_{m^2}[\psi(x) ~\vert~ \Omega^{(-\infty,R)}_{\Lambda_n}] \leq \sqrt{-2^{d+1} \log \P[Z_0 \geq R] \lambda_{\max}(G_{\Z^d, m^2} \vert_{\Lambda_{n}\times \Lambda_{n}})}.
    \end{align*} 
    To conclude, we show that $\lambda_{\max}(G_{\Z^d, m^2} \vert_{\Lambda_{n}\times \Lambda_{n}})$ is uniformly bounded by some $c(m^2, d) > 0$. To this end, recall that $G_{\Z^d, m^2}(x,y) \leq \tilde c e^{-c' |x-y|}$ for all $x,y \in \Z^d$ for appropriate constants $\tilde c(m^2, d), c'(m^2, d) > 0$—see, for instance, \cite[(1.3)]{Rod17}. In particular, we have $\sup_x \sum_{y\in \Z^d} G_{\Z^d, m^2}(x,y) \eqqcolon c(m^2, d) < \infty$. By Cauchy-Schwarz, 
    \begin{align*}
        \lambda_{\max}(G_{\Z^d, m^2} \vert_{\Lambda_{n}\times \Lambda_{n}}) &\leq \sup_{h \in \R^{\Lambda_n}: |h| = 1} \sum_{x, y \in \Lambda_n} h(x) G_{\Z^d, m^2}(x,y) h(y) \\
        &\leq \sup_{h \in \R^{\Lambda_n}: |h| = 1} \sum_{x, y \in \Lambda_n} h(x)^2 G_{\Z^d, m^2}(x,y) \leq c(m^2, d). 
    \end{align*}  
\end{proof}\vspace{-0.5 cm}
\end{proof}

We now turn to the vector-valued case.
\begin{proof}[Proof of Proposition \ref{prop:massiveGFF_existence} for $N\geq 2$]
    Let $\psi = (\psi^i(x))^{1 \leq i \leq N}_{x \in \Z^d}\sim \P_{m^2, N}$ be an $N$-vectorial $m^2$-massive GFF on $\Z^d$. By symmetry, tower property, and the FKG property of the scalar conditioned field (Proposition \ref{prop:FKG_condGFF}), we reduce to the above case: 
    \begin{align*}
        \E^{R, n}_{m^2,N} [|\psi(x)|] &\coloneqq \E_{m^2,N} \big[|\psi(x)| ~\big|~ |\psi(y)| \geq R, \; \forall y \in \Lambda_n\big] \leq 2N  \E^{R, n}_{m^2,N}\big [(\psi^1(x))_+\big] \\
        &\:= 2N \E^{R, n}_{m^2,N} \Big[  \E^{R, n}_{m^2,N}\big[(\psi^1(x))_+ ~\big|~ (\psi^j)_{j=2}^N\big] \Big] \\
        &\:= 2N \E^{R, n}_{m^2,N} \bigg[ \tilde\E_{m^2}\Big[(\tilde \psi(x))_+ ~\Big|~ |\tilde \psi(y)|^2 \geq R^2 - \sum_{j=2}^N (\psi^j(y))^2(\omega), \; \forall y \in \Lambda_n\Big] \bigg]\\
        &\:\leq 2N \E_{m^2} \big[(\psi(x))_+ ~\big|~ \psi(y) \geq R, \; \forall y \in \Lambda_n\big].
    \end{align*}
\end{proof}

\subsection{Proof of the uniqueness for small \texorpdfstring{$R$}{R} or \texorpdfstring{$I = (-\infty, b)$}{I = (-infinity, b)}}
We now address the uniqueness question and prove it in the relevant cases following the ideas in the proof of \cite[Proposition 1.7]{entropRep3+D}. The key tool is Dobrushin’s uniqueness criterion. To state it properly, for any $x \in \Z^d$ and any Gibbs state $\psi$ in $\G_1(\U^I)$, let
\begin{align*}
    p^{I, m^2}_x (\d \gamma ~\vert~ (\psi(y))_{y \neq x}) \propto e^{-\frac{2d + m^2}{2}\gamma^2 + \sum_{y: y \sim x} \gamma \psi(y)} \ind_{\{\gamma \notin I\}} \d \gamma
\end{align*}
denote the conditional law of $\psi(x)$ given $\F_{\{x\}^c} = \sigma(\psi(y): y\neq x)$. Observe that the measure $p^{I, m^2}_x (\d \gamma ~\vert~ (\psi(y))_{y \neq x})$ depends only on $$\tilde \psi(x) \coloneqq \sum_{y: y \sim x} \psi(y),$$ so, for brevity, we write $p^{I, m^2}_x(\d \gamma ~\vert~ \tilde \psi(x))$. 
Furthermore, since the law of $\psi$ belongs to $\G_1(\U^I)$ and by continuity, it follows that for every realization of $\psi$, $p^{I, m^2}_x(\d \gamma ~\vert~ \tilde \psi(x))$ has finite first moment. For $x,y \in \Z^d$, we can therefore define
\begin{align*}
    r_{m^2, I}(x, y) \coloneqq \sup\bigg\{ \frac{\W\big(p_x^{I, m^2}(\cdot | \tilde\psi(x)), p_x^{I, m^2}(\cdot | \tilde \phi(x))\big)}{|\psi(y) - \phi(y)|}: \psi, \phi \in (I^c)^{\Z^d} \text{ with } \psi(z) = \phi(z), \forall z\neq y\bigg\},   
\end{align*}
where $\W$ denotes the $1$-Wasserstein distance between two probability measures with finite first moments. In terms of these quantities, Dobrushin's uniqueness criterion can be stated as follows.
\begin{theorem}[Dobrushin's uniqueness criterion {\cite[Theorem 4 in Case 1]{Dobrushin}}] 
\label{thm:Dobrushin}
    If the following holds:
    \begin{align*}
        K(m^2, I) \coloneqq \sup_x \sum_{y: y \neq x} r_{m^2, I}(x,y) < 1,
    \end{align*}
    the set $\G_1(\U^I)$ contains at most one element.
\end{theorem}

We can now prove the uniqueness part of Proposition \ref{prop:massiveGFF_uniqueness} in the case $N=1$. Recall that in this case we consider $I=(-\infty, b)$ or $I=(a,b) \subset (-R,R)$ for $R>0$ sufficiently small. 
\begin{proof}[Proof of the uniqueness part in Proposition \ref{prop:massiveGFF_uniqueness} for $N=1$] 
    We start by noticing that for all $n$ sufficiently large such that $x \in \Lambda_{n/2}$,
    \begin{align*}
        \P_{m^2} \big[\psi(x) \in \d \gamma ~\big|~ \Omega^I_{\Lambda_n}, \F_{\{x\}^c}\big] = p^{I, m^2}_x \big(\d \gamma ~\vert~ \tilde \psi(x)\big).
    \end{align*} 
    Thus, to complete the proof of the uniqueness in this case, by Proposition \ref{prop:massiveGFF_existence} and Theorem \ref{thm:Dobrushin}, it suffices to establish that $K(m^2,I) < 1$. 
    
    In order to bound $K(m^2,I)$ appropriately, we first note that for $y\not\sim x$, $r_{m^2, I}(x,y) = 0$. To treat $x\sim y$, observe furthermore that
    \begin{align*}
        \frac{\W\big(p_x^{I, m^2}(\cdot | \tilde \psi(x)), p_x^{I, m^2}(\cdot | \tilde \phi(x))\big)}{|\psi(y) - \phi(y)|} \leq \sup_{f: L(f) \leq 1}\sup_u \Big|\frac{\d}{\d u} p^{I, m^2}_x(f | u)\Big| \frac{|\tilde \psi(x) - \tilde \phi(x)|}{|\psi(y) - \phi(y)|},
    \end{align*}
    where the outer supremum is taken over Lipschitz functions with Lipschitz constant $L(f) \leq 1$. Therefore, to finish the proof, it is enough to show that
    \begin{align*}
        \sup_x \sup_{f: L(f) \leq 1}\sup_u \Big|\frac{\d}{\d u} p^{I, m^2}_x(f | u)\Big| < \frac{1}{2d}.
    \end{align*}
    To this end, let $g: \R \rightarrow \R$ be the identity function, then
    \begin{align*}
        \Big|\frac{\d}{\d u} p^{I, m^2}_x(f | u)\Big| &= \Big|\int f(\gamma) (\gamma - p^{I, m^2}_x(g | u)) p^{I, m^2}_x(\d \gamma | u)\Big| \\
        &= \Big|\int (f(\gamma) -  f(p^{I, m^2}_x(g | u)))(\gamma - p^{I, m^2}_x(g | u)) p^{I, m^2}_x(\d \gamma | u)\Big|\\
        &\leq \int (\gamma - p^{I, m^2}_x(g | u))^2 p^{I, m^2}_x(\d \gamma | u).
    \end{align*}
    Here we used the fact that $f$ is Lipschitz with $L(f) \leq 1$ and $\int \gamma p^{I, m^2}_x(\d \gamma | u) = p^{I, m^2}_x(g | u)$. Note that the expression in the last line is nothing else than the variance of $X \sim p^{I, m^2}_x(\cdot | u)$ and recall that $\Var[X] \leq \inf_v \E[(X - v)^2]$. Let $\tilde u \coloneqq u/(2d+m^2)$, then using the explicit expression for $p^{I, m^2}_x(\cdot | u)$ and basic transformations of integrals (substitutions and integration by parts formula), we arrive at 
    \begin{align*}
        \E[(X-v)^2] &=\frac{\int_{I^c} (\gamma - v)^2 e^{-\frac{2d+m^2}{2}(\gamma - \tilde u)^2} \d \gamma}{\int_{I^c} e^{-\frac{2d+m^2}{2}(\eta-\tilde u)^2} \d \eta}\\
        &= \frac{1}{2d+m^2} + (v-\tilde u)^2 
        \hspace{-0.5mm}+\frac{(b + \tilde u - 2v) e^{-\frac{2d+m^2}{2}(b-\tilde u)^2} - (a + \tilde u - 2v)e^{-\frac{2d+m^2}{2}(a-\tilde u)^2}}{(2d+m^2) \int_{I^c - \tilde u} e^{-\frac{2d+m^2}{2}\eta^2} \d\eta}.
    \end{align*}
    Set $M(q) \coloneqq e^{-\frac{2d+m^2}{2}q^2}/\int_{I^c - \tilde u} e^{-\frac{2d+m^2}{2}\eta^2} \d\eta$. The minimum of the above expression is attained at $v= \tilde u + (M(b - \tilde u) - M(a - \tilde u))/(2d+m^2)$, yielding 
    \begin{align*}
        \Var[X] &= \frac{1}{2d+m^2} \bigg( 1 + (b-\tilde u) M(b - \tilde u) - (a-\tilde u) M(a - \tilde u) - \frac{(M(b-\tilde u) - M(a-\tilde u))^2}{2d+m^2}\bigg)\\
        &\eqqcolon \frac{1}{2d+m^2} \Big(1 + \frac{L(a,b, \tilde u, m^2)}{2d+m^2}\Big).
    \end{align*}
    
    If $I = (-\infty, b)$, that is $a= -\infty$, the terms involving $a$ in the above expression are zero, leaving us with $L = M(b-\tilde u)((2d+m^2) (b-\tilde u) - M(b-\tilde u))$. We claim that $L \leq 0$, which implies that $\Var[X] \leq \frac{1}{2d+m^2} < \frac{1}{2d}$ as desired. Indeed, if $b - \tilde u\leq 0$, the claim is obvious; else we note that 
    $$(2d+m^2) (b-\tilde u)\int_{b-\tilde u}^\infty e^{-\frac{2d+m^2}{2}\eta^2} \d\eta \leq (2d+m^2)\int_{b-\tilde u}^\infty \eta e^{-\frac{2d+m^2}{2}\eta^2} \d\eta = e^{-\frac{2d+m^2}{2}(b-\tilde u)^2}.$$ 
    In the remaining case $a, b \in (-R, R)$, observe that 
    \begin{align*}
        \frac{L(a,b, \tilde u, m^2)}{2d+m^2} &\leq (b-\tilde u) M(b - \tilde u) - (a-\tilde u) M(a - \tilde u) \leq 2R \sup_q |(q M(q))'| \\
        &\leq \frac{2R}{\int_{I^c - \tilde u} e^{-\frac{2d+m^2}{2}\eta^2} \d\eta} \leq \frac{2R}{\int_{(-R, R)^c} e^{-\frac{2d+m^2}{2}\eta^2} \d\eta}
    \end{align*}
    is arbitrarily small for all $R>0$ small enough (depending on $m^2$). In particular, there is $R_0(m^2, d) > 0$ such that for all $0<R \leq R_0(m^2,d)$, $\Var[X] < \frac{1}{2d}$ as desired.
\end{proof}

Next, we consider the vectorial case $N\geq 2$.
\begin{proof}[Proof of Proposition \ref{prop:massiveGFF_uniqueness} for $N\geq 2$]
    The above proof for $N=1$ translates easily to the $N$-vectorial case  with $I = \B^N(0, R) \subset \R^N$ up to the following changes: for vectors $\gamma, \psi \in \R^N$, interpret $\psi \gamma$ and $\psi^2$ as the inner product on $\R^N$ of $\psi$ and $\gamma$ and $\psi$ with itself, respectively; substitute $g$ with the identity on $\R^N$ and view $\frac{\d}{\d u}$ as $\nabla_u$. As in the scalar case, it only remains to show that $V \coloneqq \inf_{v \in \R^N} \E[|X-v|^2] < \frac{1}{2d}$ for any sufficiently large (depending on $N$) mass and all sufficiently small $R$ depending on this mass (more precisely, on the ratio $\frac{N}{2d+m^2}$). To this end, notice that  
    \begin{align*}
        V \leq \E[|X-\tilde u|^2] &\leq \frac{\int_{\R^N} |\gamma|^2 e^{-\frac{2d+m^2}{2}|\gamma|^2} \d \gamma}{\int_{I^c} e^{-\frac{2d+m^2}{2}|\eta|^2} \d \eta} \leq \frac{N}{2d + m^2}\frac{(2\pi/(2d+m^2))^{N/2}}{(2\pi/(2d+m^2))^{N/2} - |\B^N(0,R)|} \\
        &= \frac{N}{2d+m^2} \bigg(1- \frac{(R^2 e (2d+m^2)/N)^{N/2}}{\sqrt{\pi N}}\bigg)^{-1},
    \end{align*}
    where we used Stirling's approximation to lower-bound $\Gamma(N/2+1)$ that appears in the volume of the $N$-dimensional unit ball. In the regime where $m^2$ satisfies $\frac{N}{2d+m^2} < \tfrac{1}{2d}$, we obtain, for all $R>0$ sufficiently small (depending on the ratio $(2d+m^2)/N$ and on $d$), that $V < \frac{1}{2d}$, as desired.
\end{proof}

\subsection{Non-uniqueness for \texorpdfstring{$I= (-R,R)$}{I=(-R,R)} with \texorpdfstring{$R$}{R} sufficiently large} 
We now prove that when $N=1$ and $I = (-R, R)$ with $R>0$ large enough, the set $\G_1(\U^I)$ contains more than one measure. To this end, set $\tilde \P_n[\cdot] \coloneqq \P_{m^2}[\cdot \vert \Omega^I_{\Lambda_n}]$ and $\Q_n[\cdot] \coloneqq \P_{m^2}[\cdot \vert \Omega^I_{\Lambda_n}, \Omega^{(-\infty,R)}_{\Lambda_{2n}\setminus \Lambda_n}]$. Recall from Section \ref{subsec:massive_tightness.moment} that both of these sequences are tight, and that any of their respective subsequential limits $\tilde \P$ and $\Q$ belong to $\G_1(\U^I)$. It therefore suffices to show that $\tilde \P \neq \Q$, which we now do.

\begin{proof}[Proof of the non-uniqueness part in Proposition \ref{prop:massiveGFF_uniqueness}]
It suffices to prove that as long as $R$ is sufficiently large, $\Q[\sigma_0] \neq 0 = \tilde \E[\sigma_0]$, where $\sigma_0 = \psi(0)/|\psi(0)|$ is the sign of the field at site $0 \in \Z^d$. Note that by symmetry, $\E_{\Q_n}[\sigma_0] \geq 0 = \tilde \E_n[\sigma_0]$. Moreover, by Proposition \ref{prop:FKG_condGFF} applied to $\tilde \P_n[\cdot \vert \psi(x) \geq R, \;\forall x \in \partial_{\text{out}} \Lambda_n]$ twice, we get
\begin{align*}
    \Q_n[\sigma_0] = 2\Q_n[\psi(0) > 0] - 1 \geq 2\tilde \Q_n[ \psi(0) > 0] - 1 = \tilde \Q_n[\sigma(0)],
\end{align*}
where $\tilde \Q_n[\cdot] = \tilde \P_n[\cdot \vert \psi(x) = R, \;\forall x \in \partial_{\text{out}} \Lambda_n]$. Hence, it in turn suffices to show that for all $R$ large enough, $\E_{\tilde \Q_n}[\sigma_0]$ is uniformly bounded away from zero. But observe that the law of $(\sigma(x))_{x \in \Lambda_n}$ given $(|\psi(x)|)_{x \in \Lambda_n}$ under $\tilde \Q_n$ is that of the nearest-neighbor Ising model on $\Lambda_n$ with plus boundary condition and conductances given by $|\psi(x)||\psi(y)| \geq R^2$ for $x\sim y$ such that $x$ or $y$ is in $\Lambda_n$ (cf. \cite[Proposition 2.3]{AGS}). Let us denote the latter law by $\P^{\text{Ising}, +}_{\Lambda_n, (|\psi(x)||\psi(y)|)_{x\sim y}}$, and write $\P^{\text{Ising}, +}_{\Lambda_n, R^2}$ for the usual Ising model on $\Lambda_n$ with plus boundary condition at inverse temperature $R^2$. Then, by \cite[Theorem 3.49 and Exercise 3.31]{StatMech_book},
\begin{align*}
    \tilde \Q_n[\sigma_0] = \tilde \Q_n\Big[\tilde \Q_n[\sigma_0 ~|~ (|\psi(x)|)_{x \in \Lambda_n}]\Big] = \tilde \Q_n\Big[\E^{\text{Ising}, +}_{\Lambda_n, (|\psi(x)||\psi(y)|)_{x\sim y}}[\sigma_0]\Big] \geq \E^{\text{Ising}, +}_{\Lambda_n, R^2}[\sigma_0].
\end{align*}
The latter quantity is uniformly (in $n$) bounded from below by $1-2\delta(R)$ with $\delta(R) \searrow 0$ as $R\rightarrow \infty$, as was shown using Peierls' argument in \cite[Section 3.7.2]{StatMech_book}. Thus, the proof of $|\G_1(\U^I)| > 1$ with $I=(-R,R)$ for all sufficiently large $R$ is complete.
\end{proof}


\bibliographystyle{alpha}
\bibliography{references}

\newcommand{\etalchar}[1]{$^{#1}$}
\begin{thebibliography}{DCGR{\etalchar{+}}20}

\bibitem[AGS25]{AGS}
Juhan Aru, Christophe Garban, and Avelio Sep{\'u}lveda.
\newblock {Percolation for 2D classical Heisenberg model and exit sets of vector valued GFF}.
\newblock {\em Communications in Mathematical Physics}, 406(2):37, 2025.

\bibitem[AHPS21]{AHPS}
Michael Aizenman, Matan Harel, Ron Peled, and Jacob Shapiro.
\newblock {Depinning in integer-restricted Gaussian Fields and BKT phases of two-component spin models}.
\newblock {\em arXiv preprint arXiv:2110.09498}, 2021.

\bibitem[BD93]{critLDP}
Erwin Bolthausen and Jean-Dominique Deuschel.
\newblock {Critical Large Deviations for Gaussian Fields in the Phase Transition Regime, I}.
\newblock {\em The Annals of Probability}, 21(4):1876 -- 1920, 1993.

\bibitem[BDG01]{entropRep2D}
Erwin Bolthausen, Jean-Dominique Deuschel, and Giambattista Giacomin.
\newblock {Entropic repulsion and the maximum of the two-dimensional harmonic crystal}.
\newblock {\em The Annals of Probability}, 29(4):1670--1692, 2001.

\bibitem[BDZ95]{entropRep3+D}
Erwin Bolthausen, Jean-Dominique Deuschel, and Ofer Zeitouni.
\newblock {Entropic repulsion of the lattice free field}.
\newblock {\em Communications in Mathematical Physics}, 170(2):417 -- 443, 1995.

\bibitem[Bis20]{Biskup_notes}
Marek Biskup.
\newblock {Extrema of the Two-Dimensional Discrete Gaussian Free Field}.
\newblock In Martin~T. Barlow and Gordon Slade, editors, {\em {Random Graphs, Phase Transitions, and the Gaussian Free Field}}, pages 163--407, Cham, 2020. Springer International Publishing.

\bibitem[Dav06]{extremesGFF}
Olivier Daviaud.
\newblock {Extremes of the discrete two-dimensional Gaussian free field}.
\newblock {\em The Annals of Probability}, 34(3):962 -- 986, 2006.

\bibitem[DCGR{\etalchar{+}}20]{DCRS}
Hugo Duminil-Copin, Subhajit Goswami, Aran Raoufi, Franco Severo, and Ariel Yadin.
\newblock {Existence of phase transition for percolation using the Gaussian free field}.
\newblock {\em Duke Mathematical Journal}, 169(18), dec 2020.

\bibitem[Deu96]{Deuschel1996}
Jean-Dominique Deuschel.
\newblock {Entropic Repulsion of the Lattice Free Field, II. The 0-Boundary Case}.
\newblock {\em Communications in Mathematical Physics}, 181(3):647--665, 1996.

\bibitem[Dob70]{Dobrushin}
R.~L. Dobrushin.
\newblock {Prescribing a System of Random Variables by Conditional Distributions}.
\newblock {\em Theory of Probability \& Its Applications}, 15(3):458--486, 1970.

\bibitem[D{\"u}r09]{duerre_thesis}
Florian~Maximilian D{\"u}rre.
\newblock {\em {Self-organized critical phenomena: Forest fire and sandpile models}}.
\newblock PhD thesis, Ludwig-Maximilians-Universit{\"a}t M{\"u}nchen, June 2009.

\bibitem[Fel68]{feller1968}
William Feller.
\newblock {\em {An Introduction to Probability Theory and Its Applications. Volume II}}, volume~2.
\newblock John Wiley \& Sons, New York, 1st edition, 1968.

\bibitem[FHL24a]{FHO1}
Maximilian Fels, Lisa Hartung, and Oren Louidor.
\newblock {Gaussian free field on the tree subject to a hard wall I: Bounds}.
\newblock {\em arXiv preprint arXiv:2409.00541}, 2024.

\bibitem[FHL24b]{FHO2}
Maximilian Fels, Lisa Hartung, and Oren Louidor.
\newblock {Gaussian free field on the tree subject to a hard wall II: Asymptotics}.
\newblock {\em arXiv preprint arXiv:2409.00422}, 2024.

\bibitem[FOT10]{FukushimaOshimaTakeda0}
Masatoshi Fukushima, Yoichi Oshima, and Masayoshi Takeda.
\newblock {\em {Dirichlet Forms and Symmetric Markov Processes}}.
\newblock De Gruyter, Berlin, New York, 2010.

\bibitem[FS81]{FS}
J{\"u}rg Fr{\"o}hlich and Thomas Spencer.
\newblock {The Kosterlitz-Thouless transition in two-dimensional abelian spin systems and the Coulomb gas}.
\newblock {\em Communications in Mathematical Physics}, 81(4):527--602, 1981.

\bibitem[FV17]{StatMech_book}
Sacha Friedli and Yvan Velenik.
\newblock {\em {Statistical Mechanics of Lattice Systems: A Concrete Mathematical Introduction}}.
\newblock Cambridge University Press, 2017.

\bibitem[GS23]{GS_diff_calc}
Christophe Garban and Avelio Sep{\'u}lveda.
\newblock {Quantitative bounds on vortex fluctuations in 2d Coulomb gas and maximum of the integer-valued Gaussian free field}.
\newblock {\em Proceedings of the London Mathematical Society}, 127(3):653--708, 2023.

\bibitem[Law12]{lawler2012intersections}
G.F. Lawler.
\newblock {\em {Intersections of Random Walks}}.
\newblock Modern Birkh{\"a}user Classics. Springer New York, 2012.

\bibitem[{Lea}12]{arcsin_formula}
{Learner}(https://math.stackexchange.com/users/48763/learner).
\newblock {$P(X>0,Y>0)$ for a bivariate normal distribution with correlation $\rho$}.
\newblock Mathematics Stack Exchange, 2012.
\newblock \href{https://math.stackexchange.com/q/255400}{math.stackexchange.com/q/255400 (version: 2012-12-11)}.

\bibitem[LW16]{LW}
Titus Lupu and Wendelin Werner.
\newblock {A note on Ising random currents, Ising-FK, loop-soups and the Gaussian free field}.
\newblock {\em Electronic Communications in Probability}, 21(none), January 2016.

\bibitem[Pol75]{POLYAKOV197579}
A.M. Polyakov.
\newblock {Interaction of goldstone particles in two dimensions. Applications to ferromagnets and massive Yang-Mills fields}.
\newblock {\em Physics Letters B}, 59(1):79--81, 1975.

\bibitem[Rod17]{Rod17}
Pierre-François Rodriguez.
\newblock {A 0-1 law for the massive Gaussian free field}.
\newblock {\em Probability Theory and Related Fields}, 169, 12 2017.

\bibitem[vEL23]{vEL}
Diederik van Engelenburg and Marcin Lis.
\newblock {An elementary proof of phase transition in the planar XY model}.
\newblock {\em Communications in Mathematical Physics}, 399(1):85--104, 2023.

\end{thebibliography}

\end{document}